\documentclass[]{amsart}

\usepackage{hyperref}
\usepackage{url}

\usepackage{amsmath,amssymb,amsthm}
\usepackage{fancyvrb}
\usepackage{enumerate} 
\usepackage{cases} 
\usepackage{color} 
\usepackage{graphicx}
\usepackage{bm}
\usepackage{xcolor}
\usepackage{mathrsfs}  

\usepackage{mathtools}
\mathtoolsset{showonlyrefs}

\usepackage{geometry}

\newcommand{\subf}[2]{%
	{\small\begin{tabular}[t]{@{}c@{}}
			#1\\#2
	\end{tabular}}%
}

\numberwithin{equation}{section} \numberwithin{figure}{section}

{\theoremstyle{remark} \newtheorem{remark}{Remark}}
\newtheorem{definition}{Definition}[section]
\newtheorem{proposition}{Proposition}[section] 

\newtheorem{assumption}{Assumption}[section]
\newtheorem{theorem}{Theorem}[section]
\newtheorem{corollary}{Corollary}[section]
\newtheorem{lemma}{Lemma}[section]
{\theoremstyle{remark} \newtheorem{example}{Example}[section]}

\newcommand\pp{\partial}
\newcommand{\eps}{\varepsilon}

\newcommand{\calJ}{\mathcal{J}}
\newcommand{\calI}{\mathcal{I}}
\newcommand{\calL}{\mathcal{L}}

\newcommand{\calN}{\mathcal{N}}

\newcommand{\calV}{\mathcal{V}}
\newcommand{\calW}{\mathcal{W}}
\newcommand{\dD}{d_{\partial\Omega}}%{d_\calD}

\newcommand{\HNN}{\rm HNN}

\newcommand{\FF}{\mathscr{F}}
\newcommand{\GG}{\mathscr{G}}
\newcommand{\HH}{\mathscr{H}}

\newcommand{\N}{\mathbb{N}}
\newcommand{\R}{\mathbb{R}}
\newcommand{\Rd}{{\mathbb{R}^d}}

\newcommand{\va}{{{\bm a}}}

\newcommand{\vp}{{{\bm p}}}
\newcommand{\vq}{{{\bm q}}}

\newcommand{\vv}{{{\bm v}}}
\newcommand{\vw}{{{\bm w}}}
\newcommand{\vy}{{{\bm y}}}

\newcommand{\vphi}{{{\bm \phi}}}

\newcommand{\vpsi}{{{\bm \psi}}}
\newcommand{\vtau}{{{\bm \tau}}}
\newcommand{\vTheta}{{{\bm \Theta}}}

\newcommand\restr[2]{{% we make the whole thing an ordinary symbol
  \left.\kern-\nulldelimiterspace % automatically resize the bar with \right
  #1 % the function
  \vphantom{\big|} % pretend it's a little taller at normal size
  \right|_{#2} % this is the delimiter
  }}

\DeclareMathOperator\Div{div}

\begin{document} 
	
\title[A deep least-squares method for the obstacle problem]{A deep first-order system least squares method for the obstacle problem}
\author[G.~Acosta]{Gabriel Acosta}
\address[G.~Acosta]{Departamento de Matem\'atica, FCEyN, Universidad de Buenos Aires / IMAS, CONICET, Buenos Aires, Argentina}
\email{gacosta@dm.uba.ar}
\thanks{GA has been supported in part by PIP-2023, grant 11220220100246CO}

\author[E. Belén]{Eugenia Belén}
\address[E.~Belén]{Departamento de Matem\'atica, FCEyN, Universidad de Buenos Aires, Buenos Aires, Argentina}
\email{beleneugenia.96@gmail.com}
\thanks{}

\author[F.M.~Bersetche]{Francisco M.~Bersetche}
\address[F.M.~Bersetche]{Departamento de Matem\'atica, FCEyN, Universidad de Buenos Aires, Buenos Aires, Argentina}
\email{bersetche@gmail.com}
\thanks{FMB has been supported in part  by PIP-2023, grant 11220220100246CO}

\author[J.P.~Borthagaray]{Juan Pablo~Borthagaray}
\address[J.P.~Borthagaray]{Instituto de Matem\'atica y Estad\'istica ``Rafael Laguardia", Facultad de Ingenier\'ia, Universidad de la Rep\'ublica, Montevideo, Uruguay}
\email{jpborthagaray@fing.edu.uy}
\thanks{JPB has been supported in part by Fondo Clemente Estable grant 172393.}

\begin{abstract}
We propose a deep learning approach to the obstacle problem inspired by the first-order system least-squares (FOSLS) framework. This method reformulates the problem as a convex minimization task; by simultaneously approximating the solution, gradient, and Lagrange multiplier, our approach provides a flexible, mesh-free alternative that scales efficiently to high-dimensional settings. Key theoretical contributions include the coercivity and local Lipschitz continuity of the proposed least-squares functional, along with convergence guarantees via $\Gamma$-convergence theory under mild regularity assumptions. Numerical experiments in dimensions up to 20 demonstrate the method's robustness and scalability, even on non-Lipschitz domains.
\end{abstract}
 
\maketitle

\section{Introduction} \label{sec:introduction}

Variational inequalities arise in the modeling of several nonlinear phenomena such as contact, friction, and plasticity~\cite{chouly2023finite, kinderStam,rodrigues1987obstacle}.
A classic example is the equilibrium position of an elastic membrane constrained to remain above a given obstacle. These problems are inherently challenging for numerical approaches due to the low regularity of solutions and the a priori unknown contact boundary. While extensive work has been done on numerical methods for such problems, including finite element approaches (see, for example, ~\cite{ciarlet2002finite, glowinski2013numerical,tremolieres2011numerical}), the exploration of least-squares formulations remains relatively limited, particularly in contexts beyond finite element methods. 
We may mention \cite{Fuhrer20} as one of the first works to formalize least-squares formulations for the obstacle problem. The first-order system least-squares (FOSLS) approach proposed in that reference reformulates the obstacle problem by introducing auxiliary variables, such as the Lagrange multiplier and the gradient of the solution. This transforms the obstacle problem into a convex minimization problem subject to linear constraints. The FOSLS framework has demonstrated strong potential for providing stable approximations, facilitating error estimation, and enabling adaptivity.

In parallel, deep learning has emerged as a powerful tool for approximating solutions to partial differential equations (PDEs) by minimizing loss functionals that enforce the governing equations and boundary conditions~\cite{E18,He20,liu2,liu1,PINN,DGM18,FNM}. While not always competitive with classical methods in low-dimensional settings, neural networks show promise in high-dimensional contexts and for problems with complex geometries~\cite{weinan2022some,wojtowytsch2020can}. Importantly, neural network formulations based on first-order systems~\cite{DeepFOSLS,liu2,liu1,opschoor2024first} circumvent the need for second-order derivatives in the cost functional, streamlining computation and enabling the approximation of weak solutions.

This paper bridges the gap between the two aforementioned frameworks by proposing a deep learning approach based on the FOSLS methodology for the obstacle problem. Specifically, we adapt the Deep FOSLS framework~\cite{DeepFOSLS} to the first-order system formulation introduced in~\cite{Fuhrer20}. This approach provides a mesh-free alternative to traditional discretization methods that leverages the flexibility and scalability of neural networks. By simultaneously approximating the solution, the gradient, and the Lagrange multiplier, we demonstrate the efficiency and adaptability of this methodology in solving obstacle problems in high-dimensional instances. Moreover, the formulation within the Deep FOSLS framework enables the derivation of convergence results under appropriate regularity assumptions.

\subsection{Related work}
Recent advancements in the literature have explored the approximation of solutions to the obstacle problem through machine learning techniques. 
In \cite{cheng2023deep}, the obstacle problem is reformulated as a penalized minimization problem whose solutions converge to those of the original problem as the penalty parameter approaches infinity. The authors employ neural networks to approximate these minimizers, discretizing the cost functional through Monte Carlo integration. By decomposing the total error into three components—(i) an approximation error (dependent on network depth and width), (ii) a statistical error (governed by sample size), and (iii) an optimization error (related to the empirical loss minimization)—they establish non-asymptotic convergence rates under regularity assumptions on the penalized problem's solutions.
In \cite{zhao2022two}, the authors propose two distinct approaches for solving the obstacle problem in one and two dimensions. The first method strictly enforces the obstacle constraint through direct energy minimization, while the second employs a penalized formulation that reduces to solving a linear Dirichlet problem when the constraint is inactive. Both approaches conduct the minimization within a function space parameterized by shallow neural networks. The convergence analysis leverages the universal approximation properties of neural networks under the $L^\infty$-norm.
Reference \cite{bahja2023physics} approaches the obstacle problem within a Physics-Informed Neural Networks (PINNs) framework, minimizing the $L^2$-norm of the PDE residual while strongly enforcing boundary conditions. Similarly, \cite{DGM18} employs an analogous residual-based loss function but incorporates boundary conditions through penalization; this approach is subsequently applied to parabolic obstacle problems arising in American options pricing.
Additionally, \cite{darehmiraki2022deep} formulates the obstacle problem as a quadratic programming problem using finite element methods, subsequently solved with a neural network model.

\subsection{Organization of the paper}
The structure of this paper is as follows. Section~\ref{sec:problem_descr} introduces the reformulation of the obstacle problem as a first-order system. Section~\ref{sec:description} outlines the Deep FOSLS approach, detailing the imposition of boundary conditions and the training strategy. In Section~\ref{sec:analysis}, we establish the convergence of the method through $\Gamma$-convergence theory. 
Section \ref{sec:numerical} concludes with numerical experiments in dimensions up to 20, demonstrating the method’s scalability and robustness even on non-Lipschitz domains. 

%-----------------------------------------------------
%-----------------------------------------------------
\section{Variational Formulation and First-Order Reformulation} \label{sec:problem_descr}
%-----------------------------------------------------
%-----------------------------------------------------

This section offers a brief theoretical background on the obstacle problem and discusses the theoretical foundation for reformulating it as a first-order system, which a critical step in developing our deep learning methodology. Let $\Omega \subset \Rd$ be an open, bounded domain with Lipschitz boundary, $g \in H^1_0(\Omega)$ and $f \in L^2(\Omega)$. The obstacle problem consists in finding
\begin{equation} \label{eq:min-Dirichlet}
u_0 := \arg\min_{v \in K} I(v),    
\end{equation}
where $I$ is the classical Dirichlet energy functional
\[I \colon H^1_0(\Omega) \to \R, \quad I(v) = \frac12 \int_\Omega | \nabla v|^2 - \int_\Omega f v,
\]
and $K$ the convex set 
\[
K := \{ v \in H^1_0(\Omega) \, \colon \, v \ge g \mbox{ a.e. in } \Omega \}.
\]

It is well-known that this energy-minimization problem admits a unique solution $u_0 \in K$, and that such a solution  satisfies the variational inequality
\begin{equation} \label{eq:variational-inequality}
(\nabla u_0, \nabla (v-u_0)) \ge (f ,v-u_0) \quad \forall v \in K,
\end{equation}
 where the parentheses stand for the duality pairing in $L^2(\Omega)$. Furthermore, in the following we shall denote the $L^2(\Omega)$-norm by $\| \cdot \|$.

In the space $H^1_0(\Omega)$, we have the well-known Poincar\'e inequality
\begin{equation}\label{eq:poincare}
\| v \| \le C_P \| \nabla v \| \quad \forall v \in H^1_0(\Omega) .
\end{equation}
The constant $C_P$ only depends on the diameter of $\Omega$. This results allows one to obtain the following simple stability bound for solutions to the obstacle problem.

\begin{lemma}[$H^1$-stability]
Let $f \in L^2(\Omega)$, $g \in H^1_0(\Omega)$, and $u_0 \in K$ satisfy \eqref{eq:variational-inequality}. Then, 
\begin{equation}\label{eq:stability}
\| \nabla u_0 \| \le C \left( \| f \| + \| \nabla g \| \right), 
\end{equation}
with a constant $C = C(\Omega)$.
\end{lemma}
\begin{proof}
    We use \eqref{eq:variational-inequality} with $v = g$:
    \[
    \| \nabla u_0 \|^2 \le (\nabla u_0, \nabla g) + (f, u_0 - g) \le \frac12 \| \nabla u_0 \|^2 + \frac12 \| \nabla g \|^2 + \| f \| \|g\| + \| f \| \|u_0\| .
    \]
    For some $\epsilon >0$ to be chosen, we apply Young's and Poincar\'e's inequalities to obtain
    \[
    \| f \| \|u_0\| \le \frac{\|f\|^2}{4 \epsilon} + \epsilon \| u_0\|^2 
     \le \frac{\|f\|^2}{4 \epsilon} + \epsilon C_P^2 \| \nabla u_0 \|^2.
    \]
    Therefore, we deduce
    \[
\left( \frac12 - \epsilon C_P^2 \right) \| \nabla u_0 \|^2 \le 
\frac12 \| \nabla g \|^2 + \| f \| \|g\| + \frac{\|f\|^2}{4 \epsilon}.
    \]
    To conclude, it suffices to fix $\epsilon$ such that the term in parenthesis in the left hand side is positive and apply the Poincar\'e's inequality to $g \in H^1_0(\Omega)$.
\end{proof}

Under additional regularity hypotheses (cf. \cite{kinderStam}), the energy minimizer $u_0 \in K$ satisfies the {\em complementarity-form} equation
\begin{equation} \label{eq:obstacle}
\left\lbrace \begin{aligned}
    \min \{ -\Delta u_0 - f, \ u_0 - g \} = 0  & \mbox{ in }\Omega,\\
    u_0 = 0  & \mbox{ on } \pp\Omega.
 \end{aligned} \right.    
\end{equation}
In this work, following \cite{Fuhrer20}, we consider an equivalent reformulation of the obstacle problem as a first-order least squares system. For that purpose, we introduce two auxiliary variables: the flux density 
\[
\vphi_0 := \nabla u_0 \in L^2(\Omega; \R^d),
\]
and the Lagrange multiplier
\[
\lambda_0 := -\Delta u_0 - f \in H^{-1}(\Omega) := [H^1_0(\Omega)]'.
\]
In the following, we employ angle brackets to denote the $H^{-1}(\Omega)$-$H^1_0(\Omega)$ duality pairing.
We consider the space 
\[
\calV:= \{ \vv = (u,\vphi, \lambda) \in H^1_0(\Omega) \times L^2(\Omega; \, \R^d) \times H^{-1}(\Omega) \, \colon \, \Div \vphi + \lambda \in L^2(\Omega) \},
\]
with norm
\[
\| (u,\vphi, \lambda) \|_\calV := \left( \| \nabla u \|^2 + \| \vphi\|^2 + \| \Div \vphi + \lambda\|^2 \right)^\frac12,
\]
introduce a least-squares functional $\calJ \colon \calV \to \R$,
\begin{equation} \label{eq:def-calJ}
  \calJ(\vq; f,g) := \| \Div \vphi + \lambda + f \|^2 + \| \nabla u - \vphi \|^2 + \langle \lambda, u -g \rangle,  
\end{equation}
and the convex set
\[
K^\calV := \{ \vv = (u,\vphi, \lambda)\in \calV \, \colon \, u \ge g, \ \lambda \ge 0 \}, 
\]
and study the minimization problem
\begin{equation}\label{eq:min-FOSLS}
\vq_0 = \arg \min_{\vv \in K^\calV} \calJ(\vv; f, g).    
\end{equation}

According to \cite[Theorem 1]{Fuhrer20}, problems \eqref{eq:min-Dirichlet} and \eqref{eq:min-FOSLS} are equivalent: the unique solution of \eqref{eq:min-FOSLS} can be characterized as $\vq_0 = (u_0, \nabla u_0, -\Delta u_0 - f)$, where $u_0 \in H^1_0(\Omega)$ is the unique solution to \eqref{eq:min-Dirichlet}.

Even though the functional $\calJ$ is suitable for the analysis we pursue in this paper,
our approach makes use of Monte Carlo integration; we therefore prefer to avoid the duality pairing in $\calJ$ and replace it with suitable inner products in $L^2(\Omega)$. For the sake of rewriting $\calJ$, let us consider a triple $\vv = (u,\vphi, \lambda) \in \calV$. Because $\gamma := \Div \vphi + \lambda \in L^2(\Omega)$, we can integrate by parts to rewrite the last term in \eqref{eq:def-calJ} as
\[
\langle \lambda, u -g \rangle = \langle \gamma - \Div \vphi , u -g \rangle = (\gamma , u-g)  +  (\vphi , \nabla (u-g) ).
\]

Therefore, we introduce the space 
\[
\calW:= H^1_0(\Omega) \times  L^2(\Omega; \, \R^d) \times L^2(\Omega) \]
furnished with the product norm
\[
\| (u,\vphi, \gamma) \|_\calW := \left( \| \nabla u \|^2 + \| \vphi\|^2 + \| \gamma\|^2 \right)^\frac12,
\] the convex set
\begin{equation} \label{eq:def-kcalW}
K^\calW := \{ \vw = (u,\vphi, \gamma) \in \calW \, \colon \, u \ge g, \ \gamma \ge \Div \vphi \},
\end{equation}
and the functional $\calL \colon \calW \to \R$

\begin{equation}\label{eq:def-calL}
    \calL(\vw; f,g) := \| \gamma + f \|^2 + \| \nabla u - \vphi \|^2 + (\gamma , u-g) + (\vphi , \nabla (u-g) ) .
\end{equation}
We are interested in the minimization problem
\begin{equation}\label{eq:min-calL}
    \vp_0 = \arg \min_{\vw \in K^\calW} \calL(\vw; f, g). 
\end{equation}

It is clear that problems \eqref{eq:min-FOSLS} and \eqref{eq:min-calL} are equivalent. Indeed, we can define the bijective, norm-preserving mapping
\begin{equation} \label{eq:V-to-W}
 \calV \ni (u,\vphi,\lambda) \mapsto (u, \vphi, \Div \vphi + \lambda) \in \calW  ;
\end{equation}
we point out that the restriction of this map to $K^\calV$ is onto $K^\calW$. The equality
$\calJ(u,\vphi,\lambda) = \calL  (u, \vphi, \Div \vphi + \lambda)$ follows immediately from the definitions of both functionals.
Thus, $\calL$ has a unique minimizer in $K^\calW$ and it is given by $\vp_0 := (u_0, \vphi_0, \gamma_0),$ where  $\vq_0 = (u_0,\vphi_0,\lambda_0)$ is the minimizer of $\calJ$ in $K^\calV$, and $\gamma_0 := \lambda_0 + \Div \vphi_0$. Recalling that $\vq_0 = (u_0, \nabla u_0, -\Delta u_0 - f)$, we obtain $\vp_0 = (u_0, \nabla u_0, - f)$.

The following theorem yields the coercivity of $\calL$ over $K^\calW$.
\begin{proposition}[coercivity] \label{prop:coercivity}
    Let $\vp_0 \in K^\calW$ satisfy \eqref{eq:min-calL}. Then, there exists some $C = C(\Omega)>0$ such that 
    \[
    \calL(\vp; f,g) \ge C \| \vp - \vp_0 \|_\calW^2 \quad \forall \vp \in K^\calW.
    \]
\end{proposition}
\begin{proof}
This is a rewriting of \cite[Theorem 1]{Fuhrer20} in terms of the admissibility class $K^\calW$. Indeed, such a result states that if $\vq_0 \in K^\calV$ solves \eqref{eq:min-FOSLS}, then
\[
\calJ(\vq; f,g) \ge C \| \vq - \vq_0 \|_\calV^2 \quad \forall \vq \in K^\calV.
\]
Given $\vp = (v, \vtau, \xi) \in K^\calW,$ we use \eqref{eq:V-to-W} to introduce $\vq = (v, \vtau, \xi - \Div\vtau) \in K^\calV$ and deduce
\[
\calL(\vp; f,g) =  \calJ(\vq; f,g) \ge C \| \vq - \vq_0 \|_\calV^2 = C \| \vp - \vp_0 \|_\calW^2.
\]
\end{proof}

We next prove the local Lipschitz continuity of $\calL$ around the minimizer $\vp_0$.

\begin{proposition}[local Lipschitz continuity] \label{prop:continuity}
Let $\vp_0 \in K^\calW$ satisfy \eqref{eq:min-calL}. Then, there exists some $C = C(\Omega, f, g)>0$ such that 
    \[
    \calL(\vp; f,g) \le C (1 + \| \vp - \vp_0 \|_\calW) \ \| \vp - \vp_0 \|_\calW \quad \forall \vp \in K^\calW.
    \]
\end{proposition}
\begin{proof}
Let $\vp = (v, \vtau, \xi) \in K^\calW$. Then,
\[
\calL(\vp; f,g) =  \| \xi + f \|^2 + \| \nabla v - \vtau \|^2 + (\xi , v-g) + (\vtau , \nabla (v-g) ) .
\]
The unique minimizer of $\calL$ in $K^\calW$ is 
\[
\vp_0 := (u_0, \vphi_0, \gamma_0) =(u_0, \nabla u_0, -f) ,
\]
where $u_0 \in H^1_0(\Omega)$ is the solution of the obstacle problem, and therefore 
\[
(-\Div \vphi_0 + \gamma_0)(u_0 - g) = 0 \quad \mbox{a.e. in } \Omega.
\]
Using these properties, we can write
\begin{equation} \label{eq:bound-calL} \begin{split}
   \calL(\vp; f,g) =  & \| \xi - \gamma_0 \|^2 + \| \nabla (v-u_0) - (\vtau - \vphi_0) \|^2 \\ & + (\xi , v-g) - (\gamma_0, u_0-g) + (\Div \vphi_0, u_0-g) + (\vtau , \nabla (v-g) ). 
\end{split} \end{equation}

On the one hand, it is clear that
\[
 \| \xi - \gamma_0 \|^2 + \| \nabla (v-u_0) - (\vtau - \vphi_0) \|^2 \le C \| \vp - \vp_0 \|_\calW^2.
\]

On the other hand, we can bound
\[ \begin{split}
| (\xi , v-g) - (\gamma_0, u_0-g) | & =
| (\xi , v-u_0) + (\xi-\gamma_0, u_0-g) |
\le \| \xi \| \| v- u_0\| + \| \xi-\gamma_0 \| \|u_0-g\| \\
& \le \| \xi - \gamma_0 \| \| v- u_0\| + \| \gamma_0 \| \| v- u_0\| + \| \xi-\gamma_0 \| \|u_0-g\| \\
& \le \| \vp - \vp_0 \|_\calW^2 + ( \| \gamma_0 \| + \|u_0-g\| ) \, \| \vp - \vp_0 \|_\calW ,
\end{split}
\]
and similarly
\[
\begin{split}
    |(\Div \vphi_0, u_0-g) + (\vtau , \nabla (v-g) )| \le \| \vp - \vp_0 \|_\calW^2 + ( \| \vphi_0 \| + \|u_0-g\| ) \, \| \vp - \vp_0 \|_\calW .
\end{split}
\]
We recall $\gamma_0 = -f$ and $\vphi_0 = -\nabla u_0$. Because $u_0 \in H^1_0(\Omega)$, we can collect the estimates above and apply the Poincar\'e inequality together with \eqref{eq:stability} to conclude the desired result.
\end{proof}

%-----------------------------------------------------
%-----------------------------------------------------
\section{Description of the method} \label{sec:description}
%-----------------------------------------------------
%-----------------------------------------------------
We propose a method to approximate the unique minimizer $(u_0, \vphi_0, \gamma_0)$ of the functional in \eqref{eq:def-calL} by using neural networks. A natural, direct approach involves seeking a set of parameters $\vTheta_0 \in \R^m$ such that
\[
\calL(u_{\vTheta_0}, \vphi_{\vTheta_0}, \gamma_{\vTheta_0}) = \min_{\vTheta \in \R^m} \calL(u_{\vTheta}, \vphi_{\vTheta}, \gamma_{\vTheta}),
\]
with functions $(u_{\vTheta}, \vphi_{\vTheta}, \gamma_{\vTheta})$ belonging to a suitable neural network space. 
However, a critical challenge arises: the neural network approximations $(u_{\vTheta}, \vphi_{\vTheta}, \gamma_{\vTheta})$ are not guaranteed a priori to satisfy the admissibility constraints $ u_{\vTheta} \ge g$, $\gamma_{\vTheta} \ge \Div \vphi_{\vTheta}$ and the homogeneous Dirichlet boundary condition on $u_{\vTheta}$. 
It is a common practice in neural network approaches to weakly impose such constraints through penalty terms added to the loss functional. 
In contrast, our method enforces these conditions strongly by design, ensuring strict adherence throughout the optimization process. For boundary conditions, we adopt the strategy proposed in \cite{Berg18}, where pre-training the network to explicitly satisfy boundary data reduces the number of iterations needed for convergence.

As discussed in Section \ref{subsec:strong_b} below, we initially create certain auxiliary functions to ensure $(u_{\vTheta}, \vphi_{\vTheta}, \gamma_{\vTheta}) \in K^\calW$ for all $\vTheta \in \R^m$. The optimization procedure involves sampling $N$ points $\{x_k\}_{k=1}^N \subset \Omega$ and approximating
\[
\calL(u_{\vTheta}, \vphi_{\vTheta}, \gamma_{\vTheta}) \approx L_N(u_{\vTheta}, \vphi_{\vTheta}, \gamma_{\vTheta}),
\]
at each step of a gradient descent algorithm, where $L_N$ is defined as
\begin{equation}\label{eq:discrete_cost} \begin{split}
    	L_N (u, \vphi, \gamma) :=  \frac{|\Omega|}{N} 
	\sum_{k=1}^N \ & \big( \gamma + f(x_k) \big)^2 + \big|\nabla u(x_k) - \vphi(x_k) \big|^2 \\ & + \gamma(x_k) \big(u(x_k)-g(x_k)\big) + \vphi(x_k) \cdot \big( \nabla u(x_k)- \nabla g(x_k) \big) . 
 \end{split}
\end{equation}        
We expose the details below.

%-----------------------------------------------------
\subsection{Strong imposition of boundary and admissibility conditions}\label{subsec:strong_b}
%-----------------------------------------------------
We follow ideas from \cite{Berg18,DeepFOSLS,dist} about the imposition of the boundary conditions and restrictions. Instead of trying to directly compute either $u$ or $\gamma$ and incorporate boundary conditions or restrictions by a penalization term, we shall enforce them in the construction of the neural network approximations. For that purpose, we make use of the following notion. 

\begin{definition}[surrogate distance function] \label{def:distance-function}
Let $\Gamma \subset \overline{\partial\Omega}$ be a closed set. We say that a Lipschitz continuous function $d_{\Gamma} \colon \overline{\Omega} \to \R$ is a {\em surrogate distance function} if there exist constants $c_1,c_2 \in \R_{>0}$ such that
$c_1 d_{\Gamma}(x) \leq dist(x,\Gamma) \leq c_2d_{\Gamma}(x)$ for all $x \in \overline{\Omega}$. 
\end{definition}
Notice that the function $dist(x,\Gamma)$ is itself a surrogate distance function. However, since an explicit expression for $dist(x,\Gamma)$ is seldom available, we prefer to state our results in terms of  a computable substitute $d_{\Gamma}(x)$.
Keeping that in mind, we briefly comment on the use of surrogate distance functions in the strong imposition of admissibility and Dirichlet boundary conditions at a continuous level. Consider a Lipschitz continuous function $a \colon \R \to [0,\infty)$ such that $a(t) = 0$ if $t = 0$. In the computation of $u$ in either \eqref{eq:def-calJ} or \eqref{eq:def-calL}, we restrict the class of functions to be
\begin{equation} \label{eq:def-u}
u (x) := g(x) + \dD(x) \, a( v(x) ),
\end{equation}
where the unknown is the function $v \colon \Omega \to \R$, $g$ is the obstacle, and $\dD$ is a surrogate distance function to $\partial \Omega$.
In a similar fashion, we can incorporate the condition $\lambda \geq 0$ by considering
\begin{equation} \label{eq:def-gamma}
\gamma (x) := \Div \vphi (x) + a(\eta(x)),
\end{equation}
and the new unknown is the function $\eta \colon \Omega \to \R$. Note that, in this setting, $a(\eta)$ plays the role of the Lagrange multiplier $\lambda$.
Therefore, we can seek $\vy = (v, \vpsi, \eta)$ such that the corresponding tuple $(u,\vphi,\gamma)$, with $u$ given by \eqref{eq:def-u}, $\vpsi = \vphi$, and $\gamma $ according to \eqref{eq:def-gamma}, minimizes the loss function $\calL$. The distinction in notation for the flux variable --employing $\vpsi$ rather than $\vphi$-- serves to differentiate its role within two distinct contexts. Specifically, $\vpsi$ denotes a component of the triple $(v, \vpsi, \eta)$, which belongs to an unconstrained function space, whereas the notation $(u, \vphi, \gamma)$ is reserved for admissible triples satisfying suitable constraints.
 
%-----------------------------------------------------
\subsection{Computational aspects}
%-----------------------------------------------------

In the following, we assume the user has provided a surrogate distance  function as in Definition \ref{def:distance-function}. For a comprehensive analysis on the construction of such functions, even with higher differentiability  properties than those required in this work, we refer to \cite{dist}. The surrogate distance  can also be approximated by a neural network, as described in \cite{DeepFOSLS}. We consider a set of random points $\{x_k\}_{k=1}^N \subset \Omega$ and aim to minimize the cost functional $L_N$ introduced in \eqref{eq:discrete_cost}. The trainable parameters $\vTheta$ emerge in the computation of the auxiliary functions $v$ and $\eta$ from the construction of $u$ and $\gamma$ (see \eqref{eq:def-u} and \eqref{eq:def-gamma}).

In broad terms, the method we propose can be summarized as follows. 
Until some stop criterion is reached, do:
\begin{itemize}
\item Sample random points $\{x_k\}_{k=1}^N \subset \Omega$.
\item Update parameters: 
$ 
\vTheta \leftarrow \vTheta - \ell \nabla_{\vTheta} L_N (u_\vTheta, \vphi_\vTheta,\gamma_{\vTheta}).
$
\item Adjust learning rate $\ell$.
\end{itemize}

The computation of $L_N(u, \vphi, \eta)$ requires evaluating the derivatives of $u$ and $\vphi$ with respect to the input variables at $ \{x_k\}_{k=1}^N$.
Since all the involved functions are smooth, these derivatives can be efficiently computed by using back propagation.

For the numerical examples, we implemented our algorithm within the PyTorch library. We utilized one- to three-layer neural networks with a SoftPlus activation function for the main variables $v$ and $\vpsi$. A piecewise constant neural network was used for the variable $\eta$ and a three-layer network with a ReLU activation function was employed for high-dimensional examples. The ADAM \cite{ADAM} optimization algorithm demonstrated good results in numerical experiments. Further details about the implementation of the method can be found in Section \ref{sec:numerical}.

%-----------------------------------------------------
%-----------------------------------------------------
\section{Convergence analysis}  \label{sec:analysis}
%-----------------------------------------------------
%-----------------------------------------------------

We consider a set of unconstrained neural network functions
\begin{equation} \label{eq:def-Cm} \begin{split}
	C_m := \Big\{ (v_\vTheta,\vpsi_\vTheta, \eta_\vTheta) : \ &  v_\vTheta: \R^d \to \R, \ \vpsi_\vTheta: \R^d \to \R^d, \ \eta_\vTheta: \R^d \to \R 
 \Big\},
\end{split} \end{equation}
where we collect all the parameters in a vector $\vTheta \in \R^m$. The choice of the neural network architecture is intentionally left ambiguous at this point. In what follows, we will assume certain approximation properties of this set and later focus on a specific instance of $C_m$ and a variant from it. 

Assuming that we are able to construct a surrogate auxiliary function $\dD$ as in Definition \ref{def:distance-function}, and considering a suitable choice of $a: \R \to [0,\infty)$, we define the set of discrete admissible functions
\begin{equation} \label{eq:admissible-class-W} \begin{split}
	W_m :=  \Big\{ \vp_{\vTheta}= (u_{\vTheta},\vphi_{\vTheta}, \gamma_\vTheta) : \ &  u_{\vTheta} = g + \dD a(v_{\vTheta}), \ \vphi_{\vTheta} = \vpsi_{\vTheta}, \\ & \gamma_\vTheta = \Div \vpsi_{\vTheta} + a(\eta_\vTheta), \quad (v_{\vTheta},\vpsi_{\vTheta}, \eta_\vTheta)\in C_m \Big\},
\end{split} 
\end{equation}
which we regard as a subset of $\calW = H^1_0(\Omega) \times  L^2(\Omega; \, \R^d) \times L^2(\Omega)$ and therefore furnish with the norm $\|(u_{\vTheta},\vphi_{\vTheta}, \gamma_\vTheta)\|_{\calW} = (\| \nabla u_{\vTheta} \|^2+ \| \vphi_{\vTheta} \|^2 + \| \gamma_\vTheta \|^2 )^\frac12$. 
We remark that the fulfillment of the boundary and admissibility conditions is guaranteed within the set $W_m$, in the sense that $u_\vTheta = 0$ whenever $\dD = 0$, and moreover $u_\vTheta \ge g$, and $\gamma_\vTheta - \Div \vphi_\vTheta \ge 0$.

%-----------------------------------------------------
\subsection{$\Gamma$-convergence}
%-----------------------------------------------------

We seek  to prove the convergence of the neural network approximations obtained through our method to the minimizers of the least-squares functional $\calL$ in \eqref{eq:def-calL}. To this end, we employ $\Gamma$-convergence theory, which guarantees that if a sequence of functionals $\Gamma$-converges and their minimizers converge, then solutions to the limit problem exist, along with the convergence of both minimum values and minimizers. Below, we provide a brief review of $\Gamma$-convergence and some key results relevant to this work. We refer to \cite{braides2006} for a comprehensive treatment of the subject.

\begin{definition}[sequential $\Gamma$-convergence] \label{def:gamma-convergence}
Let $X$ be a metric space and let $F_{n}$, $F: X \to \overline{\R}$, where $\overline{\R}:= [-\infty,+\infty]$. We say that $F_{n}$ $\Gamma$-converges to $F$ (and write $F_n  \xrightarrow[]{\Gamma} F$) if, for every $x \in X$ we have
	\begin{itemize}
		\item \emph{(lim-inf inequality)} for every sequence $\{x_n\}_{n \in \N} \subset X$ converging to $x$,
		\[
		F(x) \le \liminf_{n\to\infty} F_n(x_n) ; 
		\]
		
		\item \emph{(lim-sup inequality)} there exists a sequence $\{x_n\}_{n \in \N}$ converging to $x$ such that
		\[ F(x) \ge \limsup_{n\to\infty} F_n(x_n) .\]
	\end{itemize}
\end{definition}

\begin{definition}[equi-coercivity]\label{def:equicoercividad}
	Let $\{F_n\}_{n \in \N}$ be a sequence of functions $F_n: X \to \overline{\R}$. We say that $\{F_n\}$ is equi-coercive if for all $t \in \R$ there
	exists a compact set $K_t \subset X$ such that $\{F_n \le t\} \subset K_t$.
\end{definition}

Our primary motivation for employing $\Gamma$-convergence theory stems from the following result, which states that if a sequence of functionals is equi-coercive and $\Gamma$-converges, then any corresponding sequence of minimizers converges to a minimizer of the 
$\Gamma$-limit functional. In essence, $\Gamma$-convergence, together with an appropriate compactness condition, ensures the convergence of minimizers.

\begin{theorem}[fundamental theorem of $\Gamma$-convergence]\label{teo:fund_teo_gamma_conv}
	Let $(X, d)$ be a metric space,	$\{F_n\}_{n\in\N}$ be an equi-coercive sequence of functions on $X$, and $F$ be such that $F_n  \xrightarrow[]{\Gamma} F$. Then,
	$$\exists \min_{X} F = \lim_{n \to\infty} \inf_{X} F_n.$$
	Moreover, if $\{x_n\}_{n\in\N}$ is a precompact sequence in $X$ such that $\lim_{n\to\infty} F_n(x_n) = \lim_{n\to\infty} \inf_{X} F_n$, then every
	limit of a subsequence of $\{x_n\}$ is a minimum point for $F$. 
\end{theorem}

%-----------------------------------------------------
\subsection{Convergence of the method}\label{sub:convergence}
%-----------------------------------------------------
  Our analysis is based on two generic hypotheses on the neural network functions. In Section \ref{sec:approximation-properties}, we describe specific instances of the set $C_m$ in which these assumptions are verified.

 \begin{assumption}  \label{ass:hypotheses12} We assume the following conditions. In first place, we state two conditions that we shall assume throughout the remainder of the paper.
 
 \begin{enumerate}[1)]
\item We recall that our method is based on the first-order formulation developed in Section \ref{sec:problem_descr}, which requires
$\Omega$ be a bounded, Lipschitz domain, $f \in L^2(\Omega),$  and $g \in H^1_0(\Omega)$.

\item
We additionally assume that the neural network functions in $C_m$ (cf. \eqref{eq:def-Cm}) are such that
\begin{itemize}
\item  the activation functions used for  $v_{\vTheta},\vphi_{\vTheta}$ are Lipchitz continuous;
\item the activation functions used for  $\eta_\vTheta$ are either bounded or Lipschitz continuous.
\end{itemize}
From these, to construct the classes $W_m$, we need to fix a surrogate auxiliary function $\dD$ (Definition \ref{def:distance-function}) and a Lipschitz continuous function $a: \R \to [0,\infty)$.
\end{enumerate}

Besides these two structural assumptions, we list below two conditions on the approximation capability and stability with respect to parameters of the neural network classes. We recall that the neural network class $W_m$ is deemed as a subset of $\calW = H^1_0(\Omega) \times  L^2(\Omega; \, \R^d) \times L^2(\Omega)$.

 \begin{enumerate}[1)] \setcounter{enumi}{2}
   \item Let $\vp_0 \in K^\calW$ be given by \eqref{eq:min-calL}. Then, 
$\vp_0$ can be approximated by the neural network spaces, that is, we assume that
\begin{equation}  \label{eq:hypothesis1} \tag{H1}
d(\vp_0,W_m) := \inf_{\vp_\vTheta \in W_m} \|\vp_0-\vp_\vTheta\|_\calW \to 0 \quad \mbox{as } m\to\infty.
\end{equation}

\item Let $R>0$ be fixed as in \eqref{eq:reg_func} below. We assume
\begin{equation}  \label{eq:hypothesisX} \tag{H2}
\mbox{The map $\overline{B(0,R)} \to W_m$ such that $\vTheta \mapsto \vp_\vTheta$ is continuous.}
\end{equation}
  \end{enumerate}
  \end{assumption}
  
The first two conditions are essential for our analysis, while in Section \ref{sec:approximation-properties}, we give examples of constructions of the set $C_m$ in which the latter two assumptions are verified.

Regularization plays an instrumental role in neural network approximations: one trades a small decrease in training accuracy for better generalization, thus reducing overfitting.
In the context of numerical PDEs, regularization avoids convergence towards non-physical (in our setting, non-variational), almost everywhere solutions \cite{DeepFOSLS}. Given $R>0$, that shall remain fixed, we consider $L: \R^m \to \R$ as
\begin{equation}\label{eq:reg_func}
	L(\vTheta; R) := 
	\left\lbrace
	\begin{array}{ll}
		\calL( \vp_{\vTheta}) & \mbox{if } |\vTheta| \le R, \\
		+\infty & \mbox{otherwise. } \\
	\end{array}
	\right.
\end{equation}

On the other hand, we recall that in the admissible class $W_m$ we have the representation $u_{\vTheta} = g + \dD a(v_{\vTheta}),$ $\vpsi_{\vTheta} = \vphi_{\vTheta}$,  $\gamma_\vTheta = \Div \vphi_{\vTheta} + a(\eta_\vTheta)$ with $(v_{\vTheta},\vpsi_{\vTheta},\eta_\vTheta) \in C_m$, so that we can rewrite the discrete loss functional \eqref{eq:discrete_cost} as
\begin{equation}\label{eq:discrete_cost-L}
    L_N (\vTheta; f,g) =  \frac{|\Omega|}{N} 
	\sum_{k=1}^N  \sum_{i=1}^4 G_i(\vTheta, x_k) ,
\end{equation}
 with $G_1, G_2, G_3, G_4: \R^m \times \Omega \to \R$ given by
\begin{equation}\label{def:func_aux} \begin{split}
 & G_1(\vTheta,x):=|\Div\vpsi_{\vTheta}(x) + a(\eta_\vTheta(x)) + f(x) |^2, \\
  & G_2(\vTheta,x):=| \nabla g(x) + \nabla (\dD a(v_{\vTheta}))(x) - \vpsi_{\vTheta}(x)|^2, \\ 
  & G_3(\vTheta,x):= \left[ \Div \vpsi_\vTheta (x) + a(\eta_\vTheta(x)) \right] \dD(x) a(v_{\vTheta}(x)), \\ &
  G_4(\vTheta,x) := \vpsi_\vTheta (x) \cdot  \nabla (\dD a(v_{\vTheta}))(x).
\end{split}	\end{equation}
The four functions $G_i$ correspond to the four terms in the sum \eqref{eq:discrete_cost} upon the correspondence $u \mapsto g + \dD a (v_{\vTheta})$, $\vphi \mapsto \vpsi_\vTheta$, $\gamma \mapsto \Div \vpsi_{\vTheta} + a (\eta_{\vTheta})$.

We split the proof of convergence of our method into several steps.
We start by proving the following auxiliary lemma, that shows certain boundedness properties of the functions in \eqref{def:func_aux} with respect to the neural network parameters. In the following, we restrict the domain of all neural network functions to be $\Omega$.

\begin{lemma}[boundedness] \label{lem:mayorante_integrable}
Let $G_1, \ldots, G_4$ be defined as in \eqref{def:func_aux}, with discrete admissible functions given by \eqref{eq:admissible-class-W}, and fix $R>0$ according to \eqref{eq:reg_func}. 
Then, under Assumption \ref{ass:hypotheses12}, it holds that $G_3, G_4 \in L^{\infty}(B(0,R) \times \Omega)$. Moreover, there exist $s_1, s_2 \in L^1(\Omega)$, depending on $R$, such that $|G_i(\mathbf{\Theta}, x)| \leq s_i(x)$ for all $(\mathbf{\Theta}, x) \in \overline{B(0,R)} \times \Omega$ and $i = 1,2$.
\end{lemma}
\begin{proof} Consider a triple $(v_{\vTheta},\vpsi_{\vTheta},\eta_\vTheta) \in C_m$, where $C_m$ is the set of unconstrained neural network functions, cf. \eqref{eq:def-Cm}. We begin by considering the map $ \FF \colon \R^m \times \Omega \to \R $, given by $\FF(\vTheta, x) := v_{\vTheta}(x) $.
Under Assumption \ref{ass:hypotheses12}, namely that $ v_{\vTheta} $ possesses Lipschitz continuous activation functions, it follows that $\FF$ itself is Lipschitz continuous. Consequently, $\FF$ is bounded on $ {B(0,R)} \times \Omega $, and its (weak) derivatives are likewise essentially bounded on this domain. Analogously, introducing $\GG \colon \R^m \times \Omega \to \R^m $ via $\GG(\vTheta, x) := \vpsi_{\vTheta}(x) $, we deduce that $\GG$ exhibits the same boundedness properties, with essentially bounded derivatives over ${B(0,R)} \times \Omega$.  Furthermore, considering $\HH \colon \R^m \times \Omega \to \R $ given by $\HH(\vTheta,x) := \Div\vpsi_{\vTheta}(x) + a(\eta_\vTheta(x)) $, the condition on $\eta_\vTheta$ from Assumption \ref{ass:hypotheses12} yields its essential boundedness on $ {B(0,R)} \times \Omega $. 

The function $G_4$ can be constructed using $\FF$ and $\GG$. Given that $\dD$ is a surrogate distance  function and $a$ is Lipschitz continuous, we immediately deduce the essential boundedness of $G_4$ on $ {B(0,R)} \times \Omega $. Similarly, because $G_3(\vTheta,x) = \HH(\vTheta,x) \dD(x) a(\FF(\vTheta,x))$, the boundedness of $G_3$ over the same domain follows. 

Next, the relation $G_1(\vTheta,x) = |\HH(\vTheta,x) + f(x)|^2$ yields the estimate
$$G_1(\vTheta,x) \le 2|\HH(\vTheta,x)|^2 + 2|f(x)|^2 \leq 2M + 2|f(x)|^2 =: s_1(x),$$ for some constant $M > 0$. 
Thus, the inclusion $s_1 \in L^1(\Omega)$ follows from our assumption $f \in L^2(\Omega)$. A similar argument allows us to conclude that there exists $s_2 \in L^1(\Omega)$ satisfying $G_2(\vTheta,x) \leq s_2(x)$ for all $(\vTheta,x) \in {B(0,R)} \times \Omega$.  
\end{proof}

Our goal is to approximate the minimizer $\vp_0$ of $\mathcal{L}$ in $K^\mathcal{W}$, which corresponds to the solution of the obstacle problem, using minimizers of the same functional over the set of admissible neural network functions $W_m$. However, in practice, the exact evaluation of $\calL(\vp_\vTheta)$ is unattainable due to the reliance on numerical integration. Consequently, instead of approximating $\vp_0$ through the minimizers of $\calL$ in $W_m$, the following lemma targets accounting for the error introduced in the computation of the neural network minimizers.

\begin{lemma}[approximation with $W_m$]\label{lemma:aprox-inf}
For every $m \in \N$, we define the set of neural network quasi-minimizers 
\[
\calI_m := \{ \vp_{\vTheta_m} \in W_m : \calL(\vp_{\vTheta_m}) \leq \calL(\vp_{\vTheta}) + 1/m \quad \forall \vp_{\vTheta} \in W_m \}. 
\]
Then, under Assumption \ref{ass:hypotheses12}, if $\vp_0 \in K^\calW$ is the unique minimizer of $\calL$ in $K^\calW$, we have 
\[
\sup_{\vp_{\vTheta_m} \in \calI_m}\|\vp_{\vTheta_m} - \vp_0\|_\calW \to 0
\quad  \mbox{as } m \to \infty.
\]     
\end{lemma}
\begin{proof}
Let $\vp_{\vTheta} \in W_m$ be arbitrary and $\vp_0 \in K^\calW$ be the minimizer of $\calL$.
From Propositions \ref{prop:coercivity} and \ref{prop:continuity}, we have
\begin{equation}\label{eq:funcional_eliptico}
\alpha   \| \vp_{\vTheta} - \vp_0 \|_\calW^2 \le   \calL(\vp_{\vTheta}; f,g) \le \beta (1 + \| \vp_{\vTheta} - \vp_0 \|_\calW) \ \| \vp_{\vTheta} - \vp_0 \|_\calW 
	\end{equation}
for some positive constants $\alpha, \beta$.
 
Let $\eps > 0$. By assumption \eqref{eq:hypothesis1}, we consider $m_0 > 0$ such that $d(\vp_0,W_m)<\eps$ and $1/m < \eps$ for all $m>m_0$. For every $m > m_0$, there exists $\vp_{\vTheta^*}\in W_{m}$ with $d(\vp_0, W_{m}) \ge \|\vp_{\vTheta^*}-\vp_0\|_{\calW} - \eps$, namely, $\|\vp_{\vTheta^*}-\vp_0\|_{\calW} < 2\eps$. Then, for all $m>m_0$ and every neural network quasi-minimizer $\vp_{\vTheta_m} \in \calI_m$, 
we combine the bounds in \eqref{eq:funcional_eliptico} to obtain
\[ \begin{split}
\alpha   \| \vp_{\vTheta_m} - \vp_0 \|_\calW^2 & \le \calL(\vp_{\vTheta_m}; f,g)  \\ & \le \calL(\vp_{\vTheta^*}; f,g) + \eps \le \beta  (1 + \| \vp_{\vTheta^*} - \vp_0 \|_\calW) \ \| \vp_{\vTheta^*} - \vp_0 \|_\calW + \eps < 2\beta (1 + 2\eps) \eps + \eps.
\end{split} \]
We have thus shown that for every $\eps >0$ there exists $m_0>0$ such that, for all $m>m_0$, if $\vp_{\vTheta_m} \in \calI_m$ then
\[
\| \vp_{\vTheta_m} - \vp_0 \|_\calW \le C \sqrt{\eps},
\]
and thereby the proof is concluded.
\end{proof}

The previous result is of theoretical interest, as it involves the continuous functional $\calL$. In practice, we shall make use of Monte Carlo integration and deal instead with a regularized version of \eqref{eq:discrete_cost-L}. More precisely, let $\{X_i\}_{i \in \mathbb{N}}$ be an i.i.d. sequence of random variables  on a probability space $(\Xi,\Sigma,P)$ with $X_i:\Xi \to \Omega$ $\forall i \in \mathbb{N}$, with a uniform probability density on $\Omega$. Given $\xi \in \Xi$, $R>0$ as in \eqref{eq:reg_func}, and $N \in \N$ we set $V_N(\xi) := \cup_{i \le N} \{X_i(\xi)\} $, and introduce the regularized discrete functional $L_{\xi,N}: \R^m \to \overline\R$,
\begin{equation}\label{eq:reg_func_dis}
	L_{\xi,N}(\vTheta) := 
	\left\lbrace
	\begin{aligned}
		& \frac{|\Omega|}{N}\sum_{ x \in V_N(\xi) } \sum_{i=1}^4 G_i(\vTheta,x) & \mbox{ if } |\vTheta| \le R, \\
		& +\infty & \mbox{otherwise, } \\
	\end{aligned}
	\right.
\end{equation}
where $G_1, \ldots , G_4$ are given by \eqref{def:func_aux}. 

We aim to prove the almost sure $\Gamma$-convergence of the sequence of functionals $\{L_{\xi,N}\}_{N \in \N}$ towards $L$ as the number of quadrature points $N$ tends to infinity. An auxiliary tool for this purpose is the fact that such a convergence holds in a pointwise sense.

\begin{lemma}[almost sure pointwise convergence]\label{lemma:conv_puntual}
Let $R>0$, and $L$ be defined by \eqref{eq:reg_func}. Additionally,  as in \eqref{eq:reg_func_dis}, consider the functional $L_{\xi,N}$ and $\{X_i\}_{i \in \mathbb{N}}$ an i.i.d. family of random variables on the probability space $(\Xi,\Sigma,P)$. Then, for all $\vTheta \in \R^m$, it holds that $L_{\xi,N}(\vTheta) \to L(\vTheta)$ as $N \to \infty$ $P$-almost surely. 
\end{lemma}
\begin{proof}
Because the parameter $R>0$ in the definitions of
 of $L$ and $L_{\xi,N}$ is the same, we have $L(\vTheta) = L_{\xi,N}(\vTheta) = +\infty$ whenever $|\vTheta| > R$. Therefore, there is nothing to be proven in that case and we can assume $|\vTheta| \le R$. Recalling $V_N(\xi) = \cup_{i \le N} \{X_i(\xi)\} $ with $\xi \in \Xi$, an application of the strong law of the large numbers gives
 \[
\frac{|\Omega|}{N}\sum_{ x \in V_N(\xi) } G_i(\vTheta,x)  \xrightarrow[N \to \infty]{a.s.} \int_{\Omega} G_i(\vTheta,x) \, dx .
 \]
The claim follows immediately.
\end{proof}

According to Definition \ref{def:gamma-convergence}, to prove the almost sure $\Gamma$-convergence of $\{L_{\xi,N}\}$ towards $L$ as $N \to \infty$, it suffices to show that, for every $\vTheta \in \R^m$,  the lim-inf inequality holds and that one can construct a recovery sequence satisfying the lim-sup inequality. Our argument follows the same steps as in \cite[Theorem 3.2]{DeepFOSLS}; we outline the proof for clarity.

\begin{theorem}[almost sure $\Gamma$-convergence]\label{teo:gamma_conv}
Let $R>0$, $L$ be as in \eqref{eq:reg_func}, and $L_{\xi,N}$ and $\{X_i\}_{i \in \mathbb{N}}$ be an i.i.d. family of random variables defined in the probability space $(\Xi,\Sigma,P)$ as in \eqref{eq:reg_func_dis}. Then, under Assumption \ref{ass:hypotheses12}, it holds that $L_{\xi,N} \xrightarrow[]{\Gamma} L$ as $N \to \infty$ $P$-almost surely.
\end{theorem}
\begin{proof}
Given $\vTheta \in \R^m$, we consider the trivial recovery sequence $\{\vTheta_N\}_{N \in \N} \subset \R^m$, $\vTheta_N \equiv \vTheta$. Then, Lemma \ref{lemma:conv_puntual} yields the $P$-almost sure lim-sup inequality.

As for the lim-inf inequality, we first observe that, by hypothesis \eqref{eq:hypothesisX}, if $\vTheta_N \to \vTheta$ then $\vp_{\vTheta_N} \to \vp_\vTheta$ in the norm $\| \cdot \|_{\calW}$. This, combined with the Poincaré inequality and the boundedness of $\Omega$, implies $G_i(\vTheta_{N},x) \to G_i(\vTheta,x)$ almost everywhere (up to a subsequence) in $\Omega$ ($i \in \{1,\ldots,4\}$) for the functions defined in \eqref{def:func_aux}. Additionally, from Lemma \ref{lem:mayorante_integrable}, we can bound $G_i(\vTheta_{N},x) \leq s_i(x)$, for some $s_i \in L^1(\Omega)$,  $i \in \{1,\ldots,4\}$. By combining this with Lemma \ref{lemma:conv_puntual}, we can readily reproduce the arguments presented in \cite[Theorem 3.2]{DeepFOSLS}.
\end{proof} 

Finally, we can combine the previous $\Gamma$-convergence result with a suitable form of equi-coercivity to conclude the convergence of minimizers. We emphasize that the following theorem does not assume that one is able to compute the minimum of the discrete loss functional, but instead that one can asymptotically recover such a minimizer as the number of quadrature points tends to infinity.

\begin{theorem}[convergence] \label{teo:conv} Let Assumption \ref{ass:hypotheses12} hold and, moreover, assume that, given $m \in \N$ and $R>0$, we can construct a sequence $\{\vTheta_N\}_{N \in \N} \subset \overline{B(0,R)} \subset \R^m$ such that  $\lim_{N \to \infty} L_{\xi,N}(\vTheta_N) = \lim_{N \to \infty} \inf_{\vTheta \in \R^m} L_{\xi,N}(\vTheta)$, with $L_{\xi,N}$ defined as in \eqref{eq:reg_func_dis}. Let $\vp_0 = \arg \min_{\vp \in K^\calW} \calL(\vp)$. Then, for every $\eps>0$, there $P$-almost surely exist $m_0=m_0(\eps) \in \N$, $R=R(m_0)>0$ and $N_0 = N_0(m_0) \in \N$ such that, if one constructs a sequence $\{\vTheta_N\}_{N \in \N}$ as above, then  
\[\|\vp_{N} - \vp_0 \|_{\calW} \le \eps \quad \mbox{for all } N>N_0,
\]
where $\vp_N:= (u_{\vTheta_N} , \vphi_{\vTheta_N}, \gamma_{\vTheta_N}) \in W_{m_0}$ is the neural network function defined by the parameters $\vTheta_N$.
\end{theorem}

\begin{proof}
Given $\eps > 0$, we consider the set of neural network quasi-minimizers introduced in Lemma \ref{lemma:aprox-inf}, $\calI_m = \{ \vp_{\vTheta_m} \in W_m : \calL(\vp_{\vTheta_m}) \leq \calL(\vp_{\vTheta}) + 1/m \ \forall \vp_{\vTheta} \in W_m \}$. By that lemma, there exists $m_0 > 0$ such that 
\begin{equation}\label{eq:teo_conv_1}
	\| \vp_0 - \vp_{\vTheta_{m_0}}\|_{\calW} < \eps/2 \quad \mbox{for all } \vp_{\vTheta_{m_0}} \in \calI_{m_0}. 
\end{equation}
Next, we fix $R_0>0$ large enough so that there exists a quasi-minimizer $\vp_{\vTheta^*} \in \calI_{m_0}$ with $\vTheta^* \in \overline{B(0,R_0)}$. 
Because the functional $L$ defined in \eqref{eq:reg_func} coincides with $\calL$ over $\overline{B(0,R_0)}$, the reasoning above implies that $\vp_{\vTheta_{m_0}} \in \calI_{m_0}$ for all $\vTheta_{m_0} \in \arg \min_{\vTheta \in B(0,R_0)} L(\vTheta)$.   

For this choice of $m_0$ and $R_0$, from Theorem \ref{teo:gamma_conv} we have $L_{\xi,N} \xrightarrow[]{\Gamma} L$ $P$-almost surely as $N \to \infty$. From the definition \eqref{eq:reg_func_dis} of $L_{\xi,N} \colon \R^{m_0} \to \overline\R$, it follows immediately that $\{L_{\xi , N}\}_{N \in \N}$ is an equi-coercive sequence, according to Definition \ref{def:equicoercividad}. 

Therefore, using Theorem \ref{teo:fund_teo_gamma_conv} together with the fact that every cluster point in $\{ \vTheta_N\}$ is a minimum point for $L$ and the continuity of the map $\vTheta \mapsto p_\vTheta$,
we deduce that $P$-almost surely there exists $N_0 > 0$ such that
\begin{equation}\label{eq:teo_conv_2}
\| \vp_N - \vp_{ m_0} \|_{\calW}  < \eps/2
\end{equation}
for all $N>N_0$ for some $\vp_{\vTheta_{m_0}} \in \calI_{m_0}$. 

The proof concludes by combining \eqref{eq:teo_conv_1} and \eqref{eq:teo_conv_2}.    
\end{proof}

%-----------------------------------------------------
\subsection{Approximation properties of neural networks} \label{sec:approximation-properties}
%-----------------------------------------------------
The convergence results from the previous section are based on the four conditions from Assumption \ref{ass:hypotheses12}. The first two conditions therein are {\em structural}, in the sense that they involve either the problem data or choices about the neural network activation and auxiliary functions. In contrast, the latter two express properties of the neural networks in either terms of approximation capabilities or their stability.

In this section, we show certain constructions of the set $W_m$ that satisfy the structural condition 2) from Assumption \ref{ass:hypotheses12} and, for them, we prove that hypotheses \eqref{eq:hypothesis1} and \eqref{eq:hypothesisX} hold.

\subsubsection{A first example}
For the sake of clarity, we shall restrict our analysis to fully-connected, feedforward neural networks having one hidden layer with $n$ neurons.
Accordingly, we consider the following set of unconstrained neural network functions,
\begin{equation}\label{eq:C_particular}
	C_m := \Big\{ (v_\vTheta,\vpsi_\vTheta , \eta_{\vTheta}) : v_\vTheta = B_v \sigma( A_v x + c_v ), \, \vpsi_\vTheta = B_{\vpsi} \sigma( A_{\vpsi} x + c_{\vpsi} ), \, \eta_{\vTheta} = B_\eta \sigma( A_\eta x + c_\eta ) \Big\}.
\end{equation}
Above $A_v,A_{\vpsi},A_{\eta} \in \R^{n \times d}$, $c_v,c_{\vpsi},c_\eta \in \R^{n \times 1}$, $B_v,B_\eta \in \R^{1\times n}$, $B_{\vpsi} \in \R^{d \times n }$, and $\sigma \colon \R^n \to \R^n$, and $\sigma$ a Lipschitz continuous, non-polynomial activation function, applied coordinate-wise. We gather all parameters in $\vTheta \in \R^m$, where $m = n(4d + 5)$. Note that, in this setting, stating $m \to \infty$ is equivalent to expressing that the number of neurons $n$ increases without bound. Furthermore, in this section we shall operate either under the choice
\begin{equation} \label{eq:choice-a}
a(t) := \text{ReLU}(t) = \max\{0, t\}.
\end{equation}

The approximation properties of our spaces are proved in Proposition \ref{prop:approx_prop_W_m}
below. Before addressing this issue, we note that the strong imposition of the homogenous boundary condition in  \eqref{eq:def-u}, calls for an appropriate understanding of the space to which the function $\frac{u-g}{d_{\Omega}}$ belongs.  This is the purpose of the following definition and the subsequent two lemmas; we point out that the Lipschitz condition on $\Omega$ introduced in Assumption \ref{ass:hypotheses12} is instrumental for the validity of our discussion.

\begin{definition}[weighted $L^2$ space]

Let $\dD$ be a surrogate distance function (Definition \ref{def:distance-function}). For $\alpha\ge 0$, we introduce the weighted $L^2$ space
     \[ L^2(\Omega,\alpha):= \{ w \colon   \ \dD^{\alpha}  w \in L^2(\Omega) \},\]
     equipped with the norm $\|w\|_{L^2(\Omega,\alpha)} := \| \dD^{\alpha} w \|$
  and its Sobolev extension
     \[
     H^1(\Omega,\alpha,\beta) := \{ w \in L^2(\Omega,\alpha) \colon \ \nabla w \in L^2(\Omega,\beta) \},
    \] 
 with the induced norm $\|w\|_{H^1(\Omega,\alpha,\beta)} := \|w\|_{L^2(\Omega,\alpha)} + \| \nabla w \|_{L^2(\Omega,\beta)}$.
\end{definition}
It is clear that the definition above is not dependent on the particular choice of $\dD$, as all resulting spaces turn out to be equivalent.
The following result is implicitly given in \cite{Kufner} but not stated in the form needed in this article. For completeness, we outline the proof.

\begin{lemma}[density] \label{le:densidad} The set $C^{\infty}(\overline \Omega)$ is dense in $H^1(\Omega,0,1)$.
\end{lemma}
\begin{proof} 
Taking $k=1$, $\varepsilon=2$ and $p=2$, in \cite[Proposition 9.6]{Kufner}, we first note that  $H^1(\Omega,0,1)=H^1(\Omega,1,1)$. Combining this  with the identity  $\overline{C^{\infty}(\overline \Omega)}^{\|\cdot\|_{H^1(\Omega,1,1)}}=H^1(\Omega,1,1)$ shown in \cite[Theorem 7.2]{Kufner}, we deduce the  claimed density.
\end{proof}

\begin{lemma}[embedding]
\label{le:en_espacio} For all $w\in H^1_0(\Omega)$, it holds that $w/d_{\partial \Omega}\in H^1(\Omega,0,1)$ and
\[
\| w/d_{\partial \Omega}\|_{H^1(\Omega,0,1)} \le C \| w \|_{H^1_0(\Omega)}
\]
with a constant depending on $\Omega$ and the choice of $\dD$.
\end{lemma}
\begin{proof}
Let $w\in H^1_0(\Omega)$. On the one hand, Hardy's inequality readily gives
\begin{equation} \label{eq:Hardy}
    \|w/\dD\|^2 \le C(\Omega) \| \nabla w\|^2.  
 \end{equation}
 
 On the other hand,  
    \[
    \begin{split}
    \| \dD \nabla\big( w/\dD \big)  \|^2  & = \int_\Omega \dD^2 \left( \frac{\nabla w}{\dD} - \frac{w \nabla \dD}{\dD^2} \right)^2 \\ & \le 2 \int_{\Omega} | \nabla w |^2 + \frac{w^2 |\nabla \dD|^2}{\dD^2}\\
    & \le C(\Omega, \| \nabla \dD \|_{L^{\infty}(\Omega)}) \|\nabla w\|^2,
    \end{split}
    \]
    where the last  inequality follows from \eqref{eq:Hardy} and the uniform boundedness of $|\nabla \dD|$.
\end{proof}

Our choice of $a$ in \eqref{eq:choice-a} gives rise to the following well-known result. While its proof is standard, we include it here for completeness.

\begin{lemma}[convergence of composition] \label{lem:convergence-a}
Let $a$ be given by \eqref{eq:choice-a} and let $\beta \ge 0$. If $w_m \to w$ in $H^1(\Omega,0,\beta)$, then $a(w_m) \to a(w)$ in $H^1(\Omega,0,\beta)$.
\end{lemma}
\begin{proof}
By the Lipschitz continuity of $a$, it is easy to observe that $a(w_m) \to a(w)$ in $L^2(\Omega)$:
\[
\| a(w_m) - a(w) \|_{L^2(\Omega)} \le \| w_m - w \|_{L^2(\Omega)}.
\]

To prove the convergence of the gradients, we note $\nabla (a(w)) = \chi_{\{ w > 0 \}} \nabla w$ and write
\begin{equation} \label{eq:split-gradient}
\nabla (a (w_m)) - \nabla (a(w)) =  \chi_{\{ w_m > 0 \}} \nabla (w_m - w) + \left( \chi_{\{ w_m > 0 \}} -  \chi_{\{ w > 0 \}} \right)  \nabla w.
\end{equation}
It is clear that the first term tends to $0$ in $L^2(\Omega, \beta)$:
\[
\int_\Omega |\chi_{\{ w_m > 0 \}} \nabla (w_m - w) |^2 \dD^{2\beta} \le \int_\Omega | \nabla (w_m - w) |^2 \dD^{2\beta} \to 0.
\]

For the second term in \eqref{eq:split-gradient}, since $\nabla w = 0$ a.e. on the set $\{ w = 0 \}$, we can write
\[
\int_\Omega \left| \chi_{\{ w_m > 0 \}} -  \chi_{\{ w > 0 \}} \right| \, |\nabla w|^2 \dD^{2\beta} = \int_{\{ w \neq 0 \}} \left| \chi_{\{ w_m > 0 \}} -  \chi_{\{ w > 0 \}} \right| \,  |\nabla w |^2  \dD^{2\beta} .
\]
Over the set $\{ w \neq 0 \}$ and up to a subsequence, we have $w_m \to w$ a.e., and by the continuity of $a'(t) = \chi_{\{ t > 0\}}$ for $t \neq 0$, we deduce  $\chi_{\{w_m >0 \}} \to \chi_{\{w >0 \}}$ a.e. on $\{ w \neq 0 \}$. Since the integrand is bounded by $ |\nabla w |^2  \dD^{2\beta} \in L^1(\Omega)$, we can therefore apply the dominated convergence theorem to conclude that 
\[
 \int_\Omega  \left| \chi_{\{ w_m > 0 \}} -  \chi_{\{ w > 0 \}} \right| \, |\nabla w |^2  \dD^{2\beta}  \to 0.
\]
\end{proof}

We have collected all the ingredients to show that the choice \eqref{eq:C_particular}--\eqref{eq:choice-a} gives rise to property \eqref{eq:hypothesis1}.

 \begin{proposition}[approximation] \label{prop:approx_prop_W_m}
 Let the structural condition 1) in Assumption \ref{ass:hypotheses12} hold, and the neural network classes be constructed according to \eqref{eq:admissible-class-W} by using \eqref{eq:C_particular} and \eqref{eq:choice-a}. Then, for all $\vp \in K^\calW$, we have $d(\vp,W_m) \to 0$ as $m \to \infty$; in particular, \eqref{eq:hypothesis1} holds.
\end{proposition}
\begin{proof}
First, we recall that the desired approximation property is expressed with respect to the norm $\| \cdot \|_{\calW}$, given by $\| (u, \vphi, \gamma) \|_{\calW} = \left( \| \nabla u \|^2 + \| \vphi \|^2 + \| \gamma \|^2 \right)^{1/2}$.
     
Let $\vp = (u, \vphi, \gamma ) \in K^\calW$.
  Since $ u - g \in H^1_0(\Omega) $, Lemma \ref{le:en_espacio} yields $ (u - g)/\dD \in H^1(\Omega,0,1) $. Thanks to Lemma \ref{le:densidad},  using classical results on $ H^1 $ approximation \cite[Section 6]{pinkus}, along with the fact that $ \| \cdot \|_{H^1(\Omega,0,1)} \le C(\Omega) \| \cdot \|_{H^1(\Omega)} $, we can ensure the existence of a sequence of weights  $ \{\vTheta^m_v\}_{m \in \mathbb{N}} $ such that $
\|(u - g)/\dD - v_{\vTheta^m_v}\|_{H^1 (\Omega,0,1)} \to 0 \quad \text{as } m \to \infty.$
By Lemma \ref{lem:convergence-a}, we deduce $a( v_{\vTheta^m_v}) \to a \left( (u - g)/\dD \right)$ in $H^1(\Omega,0,1)$. But,  as $ (u - g)/\dD \geq 0 $ it follows that $a(w_m) \to (u - g)/\dD$ in $H^1(\Omega,0,1)$. 
Thus,   we can guarantee the existence of a sequence $ \{\vTheta^m_v\}_{m \in \mathbb{N}} $ such that 
\begin{equation} \label{eq:int-kondratiev}
\|(u - g)/\dD - a(v_{\vTheta^m_v})\|_{H^1(\Omega,0,1)} \to 0 \quad \text{as } m \to \infty.
\end{equation}

We now aim to prove that $\|\nabla\left( u - g - \dD a(v_{\vTheta^m_v} )\right)\| \to 0$ as $m \to \infty$. To this end, we write
\[
\nabla\left( u - g - \dD a(v_{\vTheta^m_v} )\right) =  \nabla \dD \left( (u - g)/\dD - a(v_{\vTheta^m_v} ) \right) + \dD \nabla \left( (u - g)/\dD - a(v_{\vTheta^m_v} ) \right),
\]
so that 
 \[
    \begin{split}
    \| \nabla \big( u - g - \dD a(v_{\vTheta^m_v} ) \big) \| & \leq \| \nabla \dD \big( (u - g)/\dD - a(v_{\vTheta^m_v} ) \big) \| + \| \dD \nabla \big( (u - g)/\dD - a(v_{\vTheta^m_v} ) \big) \| \\
    & \leq C \| (u - g)/\dD - a(v_{\vTheta^m_v} ) \|_{H^1(\Omega,0,1)},
    \end{split}
    \]
with a constant $C>0$ depending on $\|\nabla \dD\|_{L^\infty(\Omega)}$. Using \eqref{eq:int-kondratiev}, we deduce that 
\[\|\nabla\left( u - g - \dD a(v_{\vTheta_{v}^m})\right)\| \to 0 \ \mbox{as }  m \to \infty,\]
which, naming $u_{\vTheta_{v}^m} := g + \dD a(v_{\vTheta_{v}^m})$ according to \eqref{eq:admissible-class-W}, yields $u_{\vTheta_{v}^m} \to u$ in $H^1_0(\Omega)$.

Following similar but simpler arguments, we construct the sequences $\{\vTheta_{\vpsi}^m\}_{m \in \N}$ and $\{\vTheta_{\eta}^m\}_{m \in \N}$ such that $\vpsi_{\vTheta_{\vpsi}^m} \to \vphi$ in $H^1(\Omega;\R^d)$ and $\eta_{\vTheta_{\eta}^m} \to \gamma - \Div \vphi$ in $L^2(\Omega)$. Clearly, the former yields $\|\vphi - \vpsi_{\vTheta_{\vpsi}^m}\| \to 0$ and $\| \Div \vphi - \Div \vpsi_{\vTheta_{\vpsi}^m}\| \to 0$ as $m \to \infty$. Furthermore, we can estimate
\begin{equation}\label{eq:approx_1}
    \|\gamma - \Div \vpsi_{\vTheta_{\vpsi}^m} - a(\eta_{\vTheta_{\eta}^m})\| \leq \| \gamma - \Div \vphi - a(\eta_{\vTheta_{\eta}^m}) \| + \| \Div \vphi - \Div \vpsi_{\vTheta_{\vpsi}^m} \|.
\end{equation}

Since $\vp = (u,\vphi, \gamma) \in K^\calW$, we have $\gamma - \Div \vphi \geq 0$. Additionally, as $a(t) = \max\{0,t\}$, and $\eta_{\vTheta_{\eta}^m} \to \gamma - \Div \vphi$ in $L^2(\Omega)$, we deduce that $\|\gamma - \Div \vphi - a(\eta_{\vTheta_{\eta}^m})\| \to 0$ as $m \to \infty$. Combining this with estimate \eqref{eq:approx_1}, we conclude that $\|\gamma - \Div \vpsi_{\vTheta_{\vpsi}^m} - a(\eta_{\vTheta_{\eta}^m})\| \to 0$.

By taking $\{\vTheta_m\}_{m \in \N} = \{ ( \vTheta_v^m, \vTheta_\vpsi^m, \vTheta_{\eta}^m ) \}_{m \in \N}$ and defining 
\[
\vp_{\vTheta_m} := \left( g + \dD a(v_{\vTheta_v^m}), \vpsi_{\vTheta_\vpsi^m}, \Div \vpsi_{\vTheta_\vpsi^m} + \eta_{\vTheta_\eta^m} \right),
\]
we find that $\vp_{\vTheta_m} \in W^m$ for all $m \in \N$ and $\vp_{\vTheta_m} \to \vp$ as $m \to \infty$ in $\calW$. 
\end{proof}

Next, we show the validity of \eqref{eq:hypothesisX}.

\begin{proposition}[continuity] \label{prop:cont_prop_W_m}
Let the neural network classes be constructed according to \eqref{eq:admissible-class-W} by using \eqref{eq:C_particular} and \eqref{eq:choice-a} and $R>0$ be given according to \eqref{eq:reg_func}.  Then, \eqref{eq:hypothesisX} holds: the map $\overline{B(0,R)} \to W_m$, $\vTheta \mapsto \vp_\vTheta$ is continuous.
\end{proposition}
\begin{proof}
We recall that functions in $W_m$ are constructed from neural network functions in $C_m$ by \eqref{eq:admissible-class-W},
\[
u_{\vTheta} = g + \dD a(v_{\vTheta}), \quad \vphi_{\vTheta} = \vpsi_{\vTheta}, \quad \gamma_\vTheta = \Div \vpsi_{\vTheta} + a(\eta_\vTheta), \quad (v_{\vTheta},\vpsi_{\vTheta}, \eta_\vTheta)\in C_m,
\]
and that $W_m$ is equipped from the topology from $H^1_0(\Omega) \times  L^2(\Omega; \, \R^d) \times L^2(\Omega)$.

Our choice of $C_m$ in $\eqref{eq:C_particular}$ readily implies that the map $\vTheta \mapsto (v_{\vTheta},\vpsi_{\vTheta}, \eta_\vTheta)$ is continuous with respect to the $W^{1,\infty}(\Omega) \times W^{1,\infty}(\Omega; \R^d) \times W^{1,\infty}(\Omega)$ norm.
This readily implies the continuity of the map $\vTheta \mapsto \vphi_\vTheta$ with respect to the $L^2(\Omega; \R^d)$ norm. Furthermore, by applying Lemma \ref{lem:convergence-a}, the continuity of the map  $\vTheta \mapsto \gamma_\vTheta$ with respect to the $L^2(\Omega)$ norm is straightforward.

Finally, proving the continuity of the map $\vTheta \mapsto u_\vTheta$ in $H^1_0(\Omega)$ requires showing the continuity of the correspondence $\vTheta \to \nabla(\dD (a(v_\vTheta)))$ in $L^2 (\Omega; \R^d)$. This, in turn, can be proven by splitting 
\[
 \nabla(\dD (a(v_\vTheta))) = \nabla \dD \,  a(v_\vTheta) + \dD \nabla a(v_\vTheta),
\]
using the uniform boundedness of $\dD$ and $\nabla \dD$, and Lemma \ref{lem:convergence-a}.
\end{proof}

%%%%%%%%%%%%%
 
 \subsubsection{A second example.}
In proving Proposition \ref{prop:approx_prop_W_m}, we relied on classical neural network approximation properties in the $H^1$ norm, though only $L^2$ approximation properties were essential for the term $\eta_{\vTheta}$; similarly, the continuity of the map $\vTheta \mapsto \eta_{\vTheta}$ we showed in Proposition \ref{prop:cont_prop_W_m} is only needed with respect to the $L^2$ norm. Thus, as $\eta_{\vTheta}$ involves no differential operator, its architecture necessitates only $L^2$ stability and approximation capabilities. 

We now focus on the specific construction of the neural network to compute the function $\eta_{\vTheta}$, which we recall gives rise to the approximation of the Lagrange multiplier $\lambda$ in the least squares functional \eqref{eq:def-calJ} through the relation $\lambda = a(\eta) - \Div \vphi$. More precisely,
the remainder of this section provides an analysis of $L^2$ approximation properties and stability --in the sense of hypotheses \eqref{eq:hypothesis1} and \eqref{eq:hypothesisX}, respectively-- of a certain class of Deep Neural Networks (DNN) that we use in our numerical experiments. 

Regarding the approximation property \eqref{eq:hypothesis1}, our proof proceeds in an indirect way by showing that the network is capable of generating the space of piecewise constants finite elements. This, in turn, yields, as a by-product, the $L^2$ approximation property we need. There are several by now classical results
\cite{Cybenko89, hornik1991, Barron93}
regarding the approximation properties of
neural networks, although without the
incorporation of boundary conditions.
We additionally point out to \cite{DeHaPe,HeLiXu,He_etal20, Yarotsky17} for recent results concerning upper and lower bounds for the expressive power of deep ReLU networks in the context of approximation in Sobolev spaces.
References
\cite{He_etal20,Arora18} establish the capability of networks to represent simplicial linear finite element functions,
which possess good approximation properties in the
$H^1$-norm. In  \cite{Arora18}, it is shown that at the expense of a large enough number of neurons, any continuous piecewise linear function in $\R^d$
can be written by using at most $\lceil \log_2(d+1) \rceil$ hidden layers. In \cite{HeLiXu}, on the other hand, piecewise linear finite element basis are
explicitly written in terms of  ReLU DNN functions with a cost proportional to $\lceil \log_2(\kappa_h) \rceil$, being  $\kappa_h$
the maximum number of elements sharing a nodal point. This, in turn, gives a representation of arbitrary linear finite element functions on regular meshes with an overall cost proportional to the number of nodal points. Moreover, in \cite{He_etal20}, the article  \cite{Yarotsky17} is  read in the context of hierarchical finite elements and the authors show, in particular, that finite element basis in 2d can be recovered by means of only two hidden layers.

We introduce a type of DNN that can be regarded as a ReLU DNN in which the first activation function is replaced by a step function $\sigma_s$ (i.e. $\sigma_s =H$ is the Heaviside function). Namely, functions in this particular kind of hybrid network, which hereafter we denote by $\HNN_k$, are written as
\begin{equation}
	\label{hybrid-dnn}
	f(x) = \Theta^{k}\circ \sigma \circ \Theta^{k-1} \circ \sigma \cdots \circ \Theta^1 \circ \sigma_s \circ \Theta^0(x).
\end{equation}  
Here, $\Theta^j:\R^{n_j}\to \R^{n_{j+1}}$, are affine mappings and each activation function is applied componentwise. In the neural network jargon, the number $w=\max_{1\le i \le k}\{n_{i}\}$ is usually referred to as the width of the network, the total number of neurons $s=\sum_{0\le i \le k}n_i$ is called the neural network size, and the number of hidden layers is $k$. In our case, $\R^{n_0}=\R^d$ and $\R^{n_{k+1}}=\R$.

Despite possessing good approximation properties, function spaces defined by neural networks with piecewise constant activations pose a significant optimization challenge. Standard gradient-based algorithms fail because the activation derivatives are zero almost everywhere. To circumvent this issue, we employ a variant of the Straight-Through Estimator (STE), a common heuristic for training such networks \cite{QNN}. We elaborate on the implementation of this approach in Section \ref{sec:numerical}.

Next we show, by means of elementary arguments, that for $k\ge \lceil \log_2(d+1) \rceil$, the set $\HNN_k$ contains a space of piecewise constant functions  in $\R^d$, a result that, we believe, could be of interest in itself.  Given a bounded, Lipschitz domain $\Omega\subset \mathbb{R}^d$, we consider a partition   ${\mathcal T}$ of $\Omega$ by means of closed convex polytopes, with the property
 \[
\overline\Omega=\bigcup_{\tau \in {\mathcal T}_h}\tau, \quad \tilde\tau ^o \cap \tau^o= \emptyset \ \mbox{ for all } \tilde\tau, \tau\in \mathcal T \mbox{ with } \tilde\tau\neq\tau.
\]
We are interested in exploiting the relation to standard finite element spaces with piecewise constant functions, in particular their ability to approximate functions in $L^2$ spaces. Thus, in the sequel we assume that each $\tau$ --which we will refer to as an element-- is a simplex in $\mathbb{R}^d$, namely the convex hull of $d+1$ non collinear points $\{v_0,v_1,\cdots,v_d\}$, called henceforth the vertices of $\tau$; nevertheless, our argument holds for any polytopal partition of $\Omega$. 

Following standard finite element notation, we denote the size of an element by $h_\tau=diam (\tau)$, while $ h=\max_{\tau\in \mathcal T} h_{\tau}$ stands for the maximum element size in the triangulation $\mathcal{T}$.  Given $\Omega$ and $\mathcal T$, the space $P_0$ consists of functions that take a constant value on each $\tau\in \mathcal T$. In particular, here we assume that these functions are defined in $\R^d$ considering their extension by zero outside $\Omega$.

\begin{proposition}[recovering characteristics] \label{prop:step}
Let $d\ge 1$ and $\tau$ a simplex in $\R^d$. Then, the characteristic function $\chi_\tau$ of $\tau$ can be written by means of an  $\HNN_k$ with $k\sim \mathcal{O} (\lceil \log_2(d+1)\rceil)$ and size  $s\sim  \mathcal{O}(2(d+1))$.
\end{proposition}
\begin{proof}
For $d=1$ we write $\tau=[v_0,v_1]$, where we can assume $v_0<v_1$. Obviously,
$\chi_\tau(x)=H(x-v_0)-H(x-v_1)$ where $H$ is the Heaviside function and then  $\chi_\tau\in \HNN_1$. Next, we focus on the case $d=2$, since for an arbitrary spatial dimension $d$ the argument is the same. Let $\tau$ be a triangle of vertices $\{v_0,v_1,v_2\}$. We call $L_i$ the line joining $v_i$ and $v_{i+1}$ (indices modulo $3$) and $\phi_i$ a linear function
$\phi_i:\R^2\to \R$ such that $L_i=\{x:\phi_i(x)=0\}$; we may assume that $\phi_i(x)>0$ in the interior of $\tau$. We consider $\mu_i =H \circ \phi_i$ and notice that, by construction, $\mu_i(x)=1$ for $x\in \tau$ and $i=0,1,2$, while if $x\notin \tau$ then there exists $j \in \{0,1, 2\}$ such that  $\mu_j(x)=0$. Therefore, we can write 
\[
\chi_{\tau}(x)=\min_{0\le i\le 2}\{\mu_i(x)\}.
\] 
Now, following \cite{He_etal20}, we use the fact that
 	\begin{equation}\label{min_dos}
		\min\{a,b\} = \frac{a+b}{2} -  \frac{|a-b|}{2}= \vv\cdot \mbox{ReLU} \left(W \begin{bmatrix}a \\ b \end{bmatrix}\right),
	\end{equation}
where
\[
	\vv = \frac{1}{2}\begin{bmatrix} 1 & -1 & -1 &-1 \end{bmatrix} ,\qquad W =
	\begin{bmatrix}
	1  & 1\\
	-1 &-1\\
	1  &-1\\
	-1 & 1\\
	\end{bmatrix}. \]
 Since
 \[ \min\{\mu_0(x),\mu_1(x),\mu_2(x)\}=\min\{\mu_0(x),\min\{\mu_1(x),\mu_2(x)\}\},\]
 we immediately observe that
 \[
 \chi_{\tau}(x)= \vv\cdot \mbox{ReLU} \left(W \begin{bmatrix}\mu_0(x) \\ \vv\cdot \mbox{ReLU} \left(W \begin{bmatrix}\mu_1(x) \\ \mu_2(x) \end{bmatrix}\right) \end{bmatrix}\right)
 \in \HNN_3.
 \]
 
 For a general dimension $d\ge 3$, the argument follows along the same lines by noticing that
 \[
 \min\{\mu_0(x),\mu_1(x),\ldots, \mu_d(x)\}=\min\left\{\min\{\mu_0(x)\ldots, \mu_{\lfloor (d+1)/2\rfloor-1}(x)\},\min\{\mu_{\lfloor (d+1)/2\rfloor}(x)\ldots, \mu_d(x)\} \right\}
 \]
and iterating the previous procedure. 
This construction gives rise to a binary tree with total size $s\sim 2^{\log_2(d+1)+1}$.
\end{proof}

\begin{remark} A well-known result \cite[Section 7]{pinkus}, which is similar to Proposition \ref{prop:step}, is that the characteristic function of any convex polytope can be represented by a two-hidden-layer network using the Heaviside function as an activation function. 
\end{remark}

For a given triangulation $\mathcal T$, the dimension of the space of piecewise constant functions associated to it has the same cardinality as $\mathcal T$. Thus, we arrive to the following result.

\begin{corollary}[recovering piecewise constant functions] Assume $\Omega\subset \R^d$ is a bounded polytope, and
let $\mathcal T$ be a triangulation of $\Omega$ with $N:= \# \mathcal T$. Then,  the space of piecewise constant functions on $\tau$ can be generated by an
 $\HNN_k$ with  $k\sim  (\lceil log_2(d+1)\rceil)$ and size $s\sim 2N(d+1)$. 
 \end{corollary}

Since the space of piecewise constant functions on a triangulation of $\Omega$ is capable of approximation in the $L^2$-norm, we conclude that \eqref{eq:hypothesis1} remains valid if we use $\HNN_k$ networks for the computation of the variable $\eta$.

\begin{corollary}[approximation with $\HNN_k$] \label{cor:app-HNNk}
Assume $\Omega\subset \R^d$ is a bounded polytope. Given $\eta \in L^2(\Omega)$ and $k\sim  (\lceil log_2(d+1)\rceil)$, it holds that
\[
\lim_{s \to \infty} \inf_{\eta_\vTheta \in \HNN_k} \| \eta - \eta_\vTheta \| = 0. 
\]
\end{corollary}

Our next goal is to show that the class $\HNN_k$ also allows to recover hypothesis \eqref{eq:hypothesisX}: we prove the almost everywhere continuity of elements belonging to $\HNN_k$, in the $L^2$ norm, with respect to the network parameters.
In order to state the next result, we introduce the following notation: for each element  $\vTheta = (\va,b) \in \R^{d} \times \R$, we define $f_{\vTheta} \colon \Omega \to \R$ by
$$f_{\vTheta}(x) := H(\va \cdot x + b),$$
with  $H$  the Heaviside function. We notice that if $\va=\bm{0}$ then $f_{\vTheta}(x)$ degenerates to a  constant.

\begin{lemma}[nondegeneracy] \label{lem:continuity}
Let  $\Omega \subset \R^d$ be a domain of finite measure, then 
  
\begin{enumerate}[1.]
    \item for any $\vTheta_0=(\va_0, b_0)$ with $\va_0 \neq \bm{0}$ and any sequence $\{\vTheta_{n}\}_{n \in \N} = \{(\va_n,b_n)\}_{n\in\N}$ converging to $\vTheta_0$, $f_{\vTheta_n}\to f_{\vTheta_0}$ almost everywhere in $\Omega$;
   
    \item if, in addition, $\Omega$ is bounded, then for any degenerate parameter  $\vTheta_0=(\bm{0}, b_0)$  there exists  a non-degenerate   $\vTheta' = (\va',b')$ (i.e. with $\va' \neq \bm{0}$) such that $f_{\vTheta_0} \equiv f_{\vTheta'}$ in $\Omega$.
\end{enumerate}
\end{lemma}

\begin{proof}
 We first prove {\it 1}. Assuming $\va_0 \neq \bm{0}$, we decompose $\Omega$ into three disjoint measurable sets:
\begin{align*}
\Pi^0 &:= \{x \in \Omega : \va_0 \cdot x + b_0 = 0\}, \\
\Pi^+ &:= \{x \in \Omega : \va_0 \cdot x + b_0 > 0\}, \\
\Pi^- &:= \{x \in \Omega : \va_0 \cdot x + b_0 < 0\},
\end{align*}
yielding the partition $\Omega = \Pi^0 \cup \Pi^+ \cup \Pi^-$ into disjoint sets. The non-degeneracy condition $\va_0 \neq \bm{0}$ ensures that $\Pi^0$ is contained in an affine hyperplane, and therefore has $d$-dimensional Lebesgue measure zero.

For each $x \in \Pi^+$, the strict positivity $\va_0 \cdot x + b_0 > 0$ implies, by continuity of the inner product, that there exists $N_x \in \N$ such that $\va_n \cdot x + b_n > 0$ for all $n \geq N_x$. Consequently, $f_{\vTheta_n}(x) = H(\va_n \cdot x + b_n) \to 1 = f_{\vTheta_0}(x)$ for all $x \in \Pi^+$. An identical argument shows that $f_{\vTheta_n}(x) \to 0 = f_{\vTheta_0}(x)$ for all $x \in \Pi^-$.
Since $\Pi^0$ is a null set and the convergence holds pointwise on $\Pi^+ \cup \Pi^-$, we conclude that $f_{\vTheta_n} \to f_{\vTheta_0}$ almost everywhere in $\Omega$ with respect to the Lebesgue measure.

Consider now {\it 2}. In this case,  $\vTheta_0 = (\bm{0}, b_0)$. Without loss of generality, we may assume $b_0 \geq 0$, which immediately gives $f_{\vTheta_0} \equiv 1$ over $\Omega$. Since $\Omega$ is bounded, we can construct an explicit non-degenerate counterpart $\vTheta' = (\bm{1}, b')$ by choosing $\bm{1} = (1,\ldots,1) \in \R^d$ and $b' > \sup_{x \in \Omega} \|\bm{1} \cdot x\|_{L^\infty(\Omega)}$. This guarantees the strict positivity
\[
\bm{1} \cdot x + b' > 0 \quad \text{for all } x \in \Omega,
\]
and consequently $f_{\vTheta'} \equiv 1$. Thus, the equivalence $f_{\vTheta_0} \equiv f_{\vTheta'}$ holds in $\Omega$.  \end{proof}

The next theorem establishes the $L^2$-continuity of any $\HNN_k$-type neural network with respect to the parameters.
\begin{theorem}[almost everywhere continuity]\label{teo:ae_continuity}
    Consider a sequence $\{\vTheta_n\}_{n\in\N} \subset \R^m$ converging to some $\vTheta_0 \in \R^m$, where $v_{\vTheta_n}$ and $v_{\vTheta_0}$ are $\HNN_k$-type neural networks parameterized by $\vTheta_n$ and $\vTheta_0$, respectively. For any domain $\Omega \subset \R^d$ with finite measure, there exists a set $\calN \subset \R^m$  such that:
\begin{enumerate}[1.]
    \item if $\vTheta_0 \notin \calN$, then $v_{\vTheta_n} \to v_{\vTheta_0}$ in $L^2(\Omega)$;
    \item if $\Omega$ is bounded, each $\vTheta \in \calN$ admits $\vTheta' \notin \calN$ satisfying $v_{\vTheta} \equiv v_{\vTheta'}$ in $\Omega$.
\end{enumerate}
Moreover, $\calN$ is a finite union of subspaces of  $\R^m$, each one of dimension $m-d$  (in particular $\calN$ is a null set with respect to the Lebesgue measure).
\end{theorem}

\begin{proof}
For any $\vTheta \in \R^m$, let $l_{\vTheta} \colon \Omega \to \R^{\ell}$ denote the first layer of an $\HNN_k$-type neural network, namely
\[
l_{\vTheta}(x) := \big(H( \bm a_1 \cdot x  + b_1), \ldots, H( \bm a_\ell \cdot x + b_\ell)\big),
\]
where $\bm a_i \in \R^d$ and $b_i \in \R$ for $i = 1,\ldots,\ell$. For $1 \leq i \leq \ell$, consider the sets of singular parameters,
\[
\calN_i := \{\vTheta \in \R^m : \bm a_i = \bm 0\}.
\]
We note that each $\calN_i $ is a  proper affine subspace of dimension $m-d$ and define  
$$\calN := \bigcup_{i=1}^\ell \calN_i.$$  

For $\vTheta_0 \in \R^m \setminus \calN$, Lemma~\ref{lem:continuity} guarantees the componentwise convergence $l_{\vTheta_n}(x) \to l_{\vTheta_0}(x)$ for almost every $x \in \Omega$. As the remaining layers of the $\HNN_k$ architecture consist of continuous activation functions, the almost everywhere convergence is preserved, that is
\begin{equation}
 \label{eq:aeConv}
 v_{\vTheta_n}(x) \to v_{\vTheta_0}(x) \quad \text{for a.e. } x \in \Omega.
\end{equation}

With the same procedure as in the proof of Lemma~\ref{lem:mayorante_integrable}, we first establish the uniform boundedness property
\[
\sup_{\|\vTheta\| \leq R} \|v_{\vTheta}\|_{L^\infty(\Omega)} < \infty \quad \text{for any } R > 0.
\]

Now, if  $\vTheta_n \to \vTheta_0$, consider the compact parameter set $\{\vTheta_0\} \cup \{\vTheta_n\}_{n\in\mathbb{N}}$, and let $R$ denote its radius. The preceding estimate yields a uniform bound $M > 0$ satisfying
\[
\max\left(\|v_{\vTheta_0}\|_{L^\infty(\Omega)}, \ \sup_{n\in\mathbb{N}} \|v_{\vTheta_n}\|_{L^\infty(\Omega)}\right) \leq M.
\]
Since $\Omega$ has finite measure, the constant function $M$ is integrable and therefore, formula \eqref{eq:aeConv} together with the dominated convergence theorem yield
\[
\lim_{n\to\infty} \|v_{\vTheta_n} - v_{\vTheta_0}\| = 0, 
\]
and the first assertion of this theorem follows.

When $\Omega$ is bounded, any parameter $\vTheta \in \calN$ must satisfy $\bm a_i = \bm 0$ for some non-empty subset of indices $I \subseteq \{1,\ldots,\ell\}$. Through Lemma~\ref{lem:continuity}, we may construct modified parameters by replacing each zero-weight pair $(\bm a_i, b_i)_{i\in I}$ with non-degenerate alternatives $(\bm a_i', b_i')_{i\in I}$, where $\bm a_i' \neq \bm 0$, to a new parameter $\vTheta'$.
This construction preserves the layer mapping exactly:
\[
l_{\vTheta}(x) = l_{\vTheta'}(x) \quad \forall x \in \Omega.
\]
Consequently, the full network implementations coincide:
\[
v_{\vTheta}(x) = v_{\vTheta'}(x) \quad \forall x \in \Omega ,
\]
with the key property that $\vTheta' \notin \calN$ avoids the singular set entirely.
\end{proof}

\begin{remark}
Theorem~\ref{teo:ae_continuity} establishes that the mapping $\vTheta \mapsto \eta_{\vTheta}$ is continuous almost everywhere in the parameter space. This result yields two key consequences for networks with step activation functions:

\begin{enumerate}
    \item Hypothesis \eqref{eq:hypothesisX} holds for almost every parameter configuration $\vTheta \in \R^m$;
    
    \item Restricting to parameters within the continuity set does not compromise the expressive capacity of the architecture, as each realizable function $\eta_{\vTheta}$ admits an equivalent representation $\eta_{\vTheta'}$ with $\vTheta'$ in the continuity set.
\end{enumerate}
Notably, the second point guarantees the validity of hypothesis \eqref{eq:hypothesis1} even when the parameter space is restricted to exclude the negligible set of degenerated points.
\end{remark}

\begin{remark}
While we stated Corollary \ref{cor:app-HNNk} and Theorem \ref{teo:ae_continuity} in the $L^2$ norm, which is the one we are interested in, it is clear that the same results are valid if one considers the $L^p$ norm for an arbitrary $p \in [1,\infty)$.
\end{remark}

%-----------------------------------------------------
%-----------------------------------------------------
\section{Computational experiments} \label{sec:numerical}
%-----------------------------------------------------
%-----------------------------------------------------

In this section, we present numerical results for the method proposed in Section \ref{sec:description}. We did not prioritize any specific neural network architecture, employing two- and three-layer fully connected networks with SoftPlus activation functions to construct $u_{\vTheta}$ and $\vphi_{\vTheta}$. For the construction of $\eta_{\vTheta}$, we used one- to three-layer networks with ReLU or step activation functions. Networks with more layers demonstrated better performance in high-dimensional problems. Across all experiments, we utilized $a(t) := t^2$, which has exhibited satisfactory performance. During training, the well-known ADAM algorithm \cite{ADAM} was employed to update the parameters, with the initial learning rate scheduled to decrease linearly to zero by the final iteration.

When step activation functions are used in the construction of $\eta_{\vTheta}$, gradient computation during training relies on the Straight-Through Estimator \cite{STE}. In this approach, the derivative of the Heaviside activation function, which vanishes almost everywhere, is replaced by the approximation $\frac{1}{c}\mathbb{I}_{[0,c]}(x)$ for a prescribed parameter $c>0$.

Although the convergence theory established in Section \ref{sec:analysis} employs the functional $\calL$, defined in \eqref{eq:def-calL}, our experiments show very similar results using both $\calL$ and $\calJ$, defined in \eqref{eq:def-calJ}. That is, we interchangeably used the discrete versions of these functionals, with $L_N$ being the discrete functional defined in \eqref{eq:discrete_cost}, and $J_N$ being the discrete functional defined as follows:
\[ \begin{split}
    	J_N (u, \vphi) :=  \frac{|\Omega|}{N} 
	\sum_{k=1}^N \ & \big( \Div{\vphi}(x_k) + \lambda(x_k) + f(x_k) \big)^2 + \big(\nabla u(x_k) - \vphi(x_k) \big)^2 \\ & + \lambda(x_k) \big(u(x_k)-g(x_k)\big). 
 \end{split}
\]
For this functional, we can compute $u_\vTheta$ and $\vphi_\vTheta$ as in \eqref{eq:admissible-class-W}, and we can simply use  $a(\eta_{\vTheta})$ to approximate $\lambda$.

The code used to implement the experiments in this section is available at Github \footnote{ \url{https://github.com/fbersetche/Deep-FOSLS-for-the-Obstacle-Problem-} }

%-----------------------------------------------------
\begin{example}[A high-dimensional problem]\label{sec:example1}
%-----------------------------------------------------
We consider the following high-dimensional instance of problem \eqref{eq:obstacle}. Let $\Omega = B_{R_0}(0) = \{ x \in \Rd : | x | < R_0 \}$ with $R_0 \in \R_{>0}$, and let $f_d \colon \R \to \R$ be such that $-\Delta f_d(|x|) = \delta_0$, that is, the $d$-dimensional fundamental solution of the Laplace operator, as a function of the radius, centered at $x = 0$. Let $0 < r_0 < R_0$, and define the polynomial $Q \colon \R \to \R$ as $Q(r) = ar^4 + br^2 + c$, with coefficients determined to satisfy the conditions
\begin{equation*}
	Q(r_0) = f_d(r_0) - f_d(R_0),\qquad
	Q'(r_0) = f'_d(r_0),\qquad
	Q(R_0) = 0.
\end{equation*}
By taking $g \colon \R^d \to \R$ as $g(x) := Q(|x|)$, the solution to problem \eqref{eq:obstacle} is
\begin{equation*}
u(x) = \left\lbrace \begin{aligned}
	Q(|x|) & \quad \text{ if } |x| \leq r_0, \\
	  f_d(|x|) - f_d(R_0)  & \quad \text{ if } r_0 < |x| \leq R_0.\\
\end{aligned} \right.
\end{equation*}

We first test our approach in a 10-dimensional domain ($d=10$). Figure \ref{fig:dim2} displays the results we obtained by constructing $u_{\vTheta}$ and $\vphi_{\vTheta}$ using 3-layer neural networks with 100 SoftPlus activation functions per layer, and a similar architecture for $\eta_{\vTheta}$, but using ReLU activation functions instead of SoftPlus. The SoftPlus activation function was used with parameter $\beta = 100$. At the end of the optimization algorithm, we computed the average value $L_N (\vTheta) = 5.193$ and an average $L^2$-error of 0.2705.

Figure \ref{fig:dim2_k2} corresponds to $d=20$, where we used the same architecture with 150 neurons per layer. We observed fast convergence in the number of iterations, reaching an average $L_N (\vTheta) = 9702.4$ and an average $L^2$-error of 8.155 by the end of the minimization algorithm.

\begin{figure}[h]
	\centering
	\begin{tabular}{|c|c|c|c|}
		\hline
		\subf{\includegraphics[width=49mm]{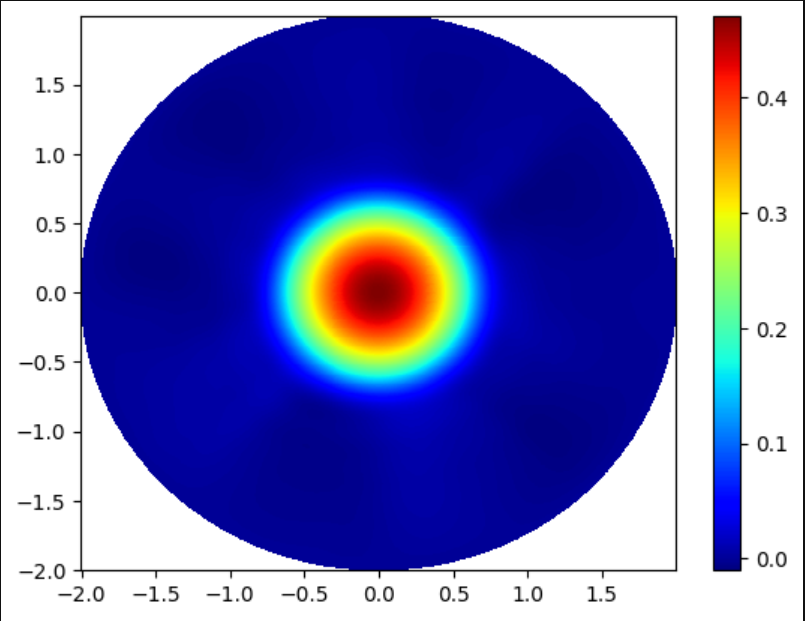}}
		{$ \restr{u_{\vTheta}}{\{x_3,\ldots,x_{10} = 0\}}$.}
		&
		\subf{\includegraphics[width=50mm]{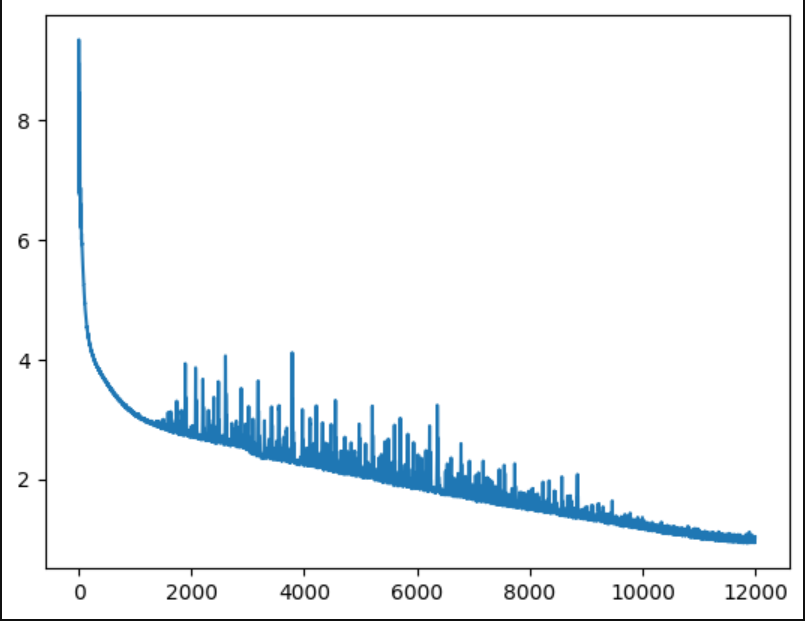}}
        {$\log(L_{N}(\vTheta))$ vs. iterations.}
		\\
		\hline
		\subf{\includegraphics[width=50mm]{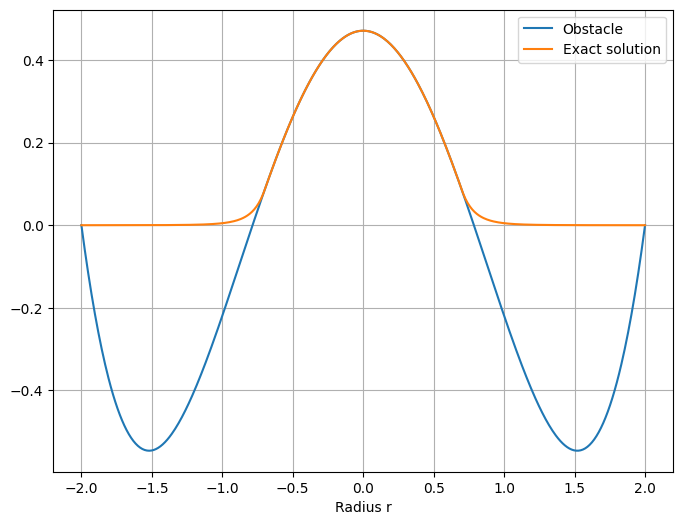}}
		{ Exact solution $u$ and obstacle $g$. }
		&
		\subf{\includegraphics[width=50mm]{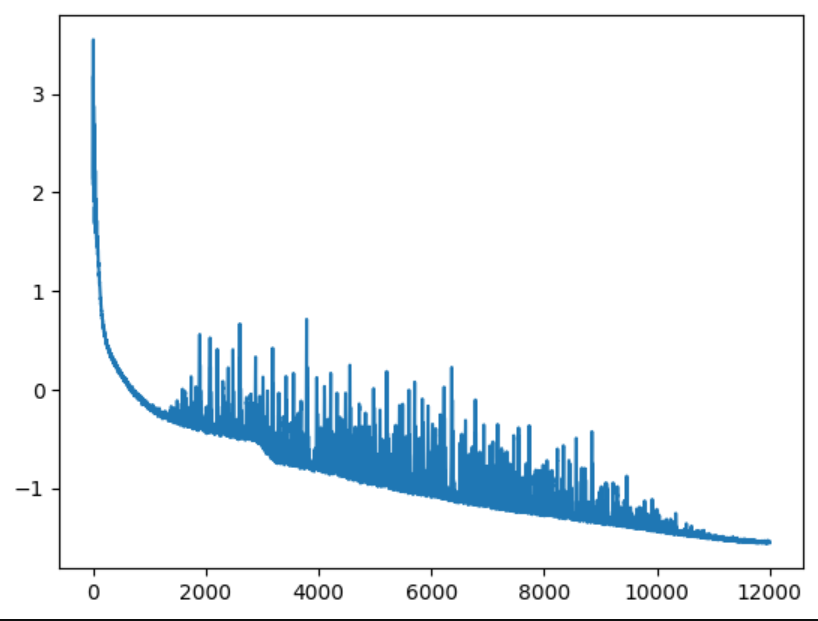}}
		{ $ \log \|u_{\vTheta} - u\|  $ vs. iterations.}
		\\
		\hline
	\end{tabular}
	\caption{Top left: A slice of the computational solution $u_{\vTheta}$ to example \eqref{sec:example1} in the case $d = 10$, $r_0 = 0.7$, and $R_0 = 2$. In the computation, we used a learning rate $\ell = 0.001$, with 20000 collocation points, 10000 optimization steps, and 65112 degrees of freedom. The bottom left panel exhibits the obstacle and the exact solution as a function of $r =|x|$. We also report the evolution of the loss function (top right) and the $L^2$-error (bottom right). The $L^2$-error was estimated using Monte Carlo integration, with a sample of 40000 points.}
	\label{fig:dim2}
\end{figure}

\begin{figure}[h]
	\centering
	\begin{tabular}{|c|c|c|c|}
		\hline
		\subf{\includegraphics[width=50mm]{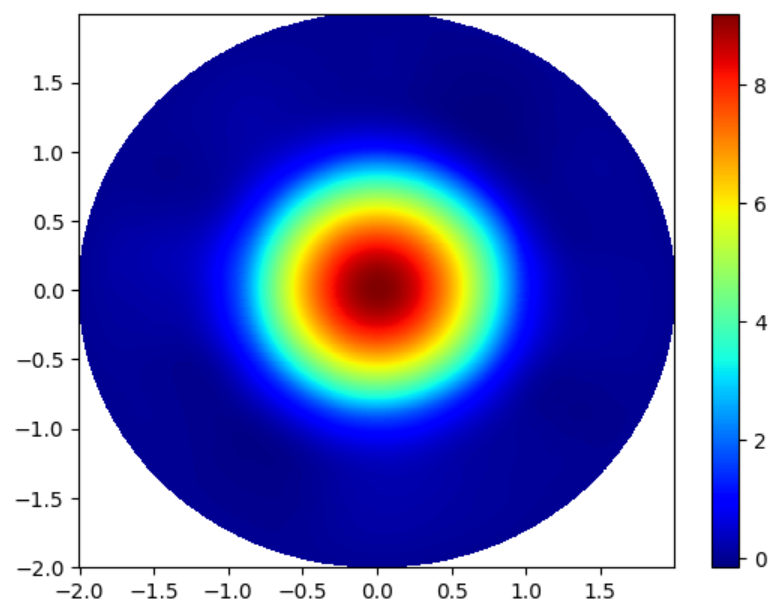}}
		{$ \restr{u_{\vTheta}}{\{x_3,\ldots,x_{20} = 0\}}$.}
		&
		\subf{\includegraphics[width=50mm]{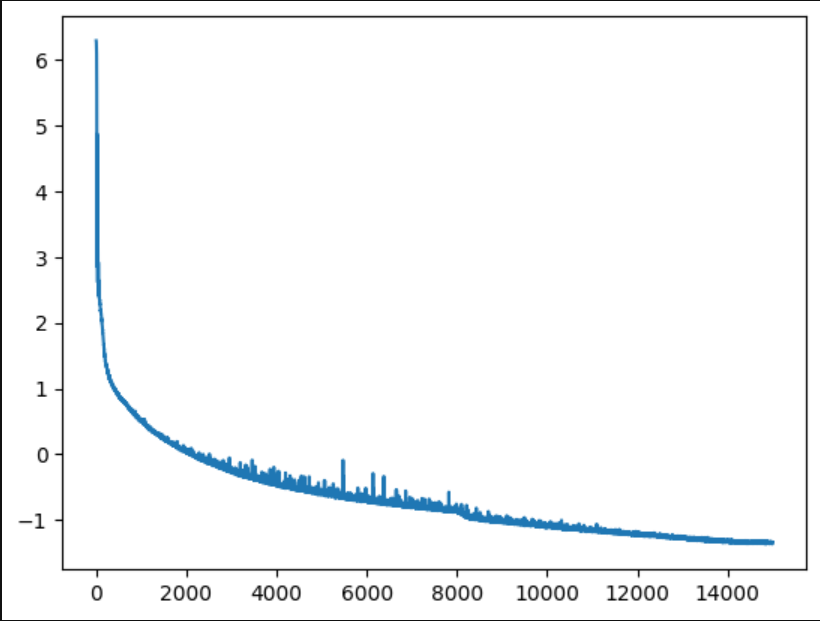}}
        {$\log(L_{N}(\vTheta))$ vs. iterations.}
		\\
		\hline
		\subf{\includegraphics[width=50mm]{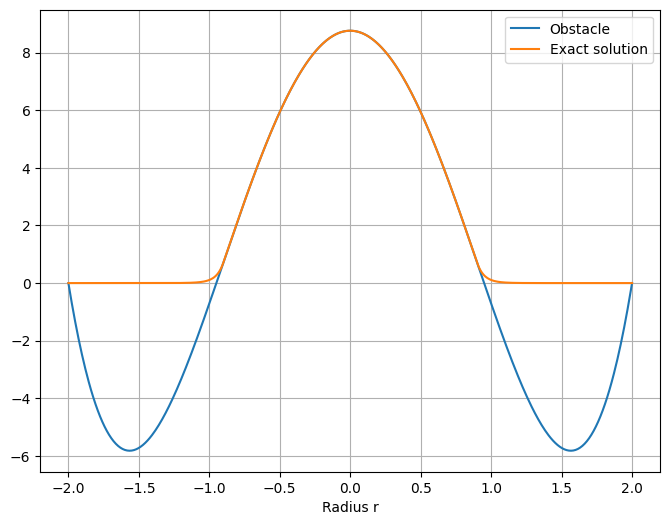}}
		{ Exact solution $u$ and obstacle $g$. }
		&
		\subf{\includegraphics[width=50mm]{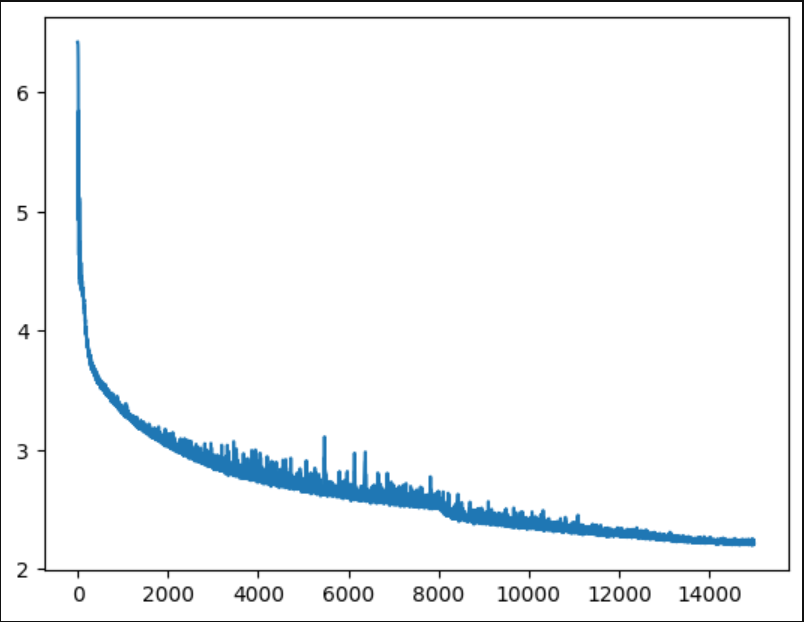}}
		{  $ \log \|u_{\vTheta} - u\|  $ vs. iterations.}
		\\
		\hline
	\end{tabular}
	\caption{A slice of the computational solution $u_{\vTheta}$ (top left), evolution of the loss function (top right), the exact solution and the obstacle as a function of $r = \|x\|_2$ (bottom left), and evolution of the $L^2$-error (bottom right) for \eqref{sec:example1} with $d = 20$, $r_0 = 0.9$, and $R_0 = 2$. We used a learning rate $\ell = 0.001$, with 20000 collocation points, 14000 optimization steps, and 148672 degrees of freedom.}
	\label{fig:dim2_k2}
\end{figure}
\end{example}

%-----------------------------------------------------
\begin{example}[A problem on a slit domain] \label{ex:singularly-perturbed}
%-----------------------------------------------------
We consider the non-Lipschitz domain $\Omega = B_{1}(0) \setminus \{0\} \times [0,1]$, with $f \equiv 0$, and the a two-peak obstacle function $g: \Omega \to \R$,
\[
g(x) = \max\{ 10e^{-30 |x - (-0.4,-0.5)|^2}-1, \, 0 \} + \max\{ 15e^{-30 |x - (-0.4,0.5)|^2}-1, \, 0 \}. 
\]
Figures \ref{fig:sing_pert} and \ref{fig:sol_vs_obs} exhibit our computed solutions for this instance of problem \eqref{eq:obstacle}. In this case, we observe a fast convergence towards the solution, reaching $L_N (\vTheta) = 42.46$. We employed a two-layer architecture with 150 neurons per layer and a Softplus activation function with parameter $\beta = 100$ for both $u_\vTheta$ and $\phi_\vTheta$. For $\eta_\vTheta$, we employ a piecewise constant architecture comprising a three-layer $\HNN_3$, where the first and third layers use the ReLU activation function, and the second layer adopts the step activation function (see Proposition \ref{prop:step}). During training, the STE is used to approximate the gradient with respect to the parameters of $\eta_{\vTheta}$ by substituting the derivative of the Heaviside activation functions with $\frac{1}{c}\mathbb{I}_{[0,c]}(x)$, where $c = 0.5$.
Our results indicate that the method performs robustly even on non-Lipschitz domains, which lie outside the scope of the theoretical framework.

\begin{figure}[h]
	\centering
	\begin{tabular}{|c|c|c|c|}
		\hline
		\subf{\includegraphics[width=50mm]{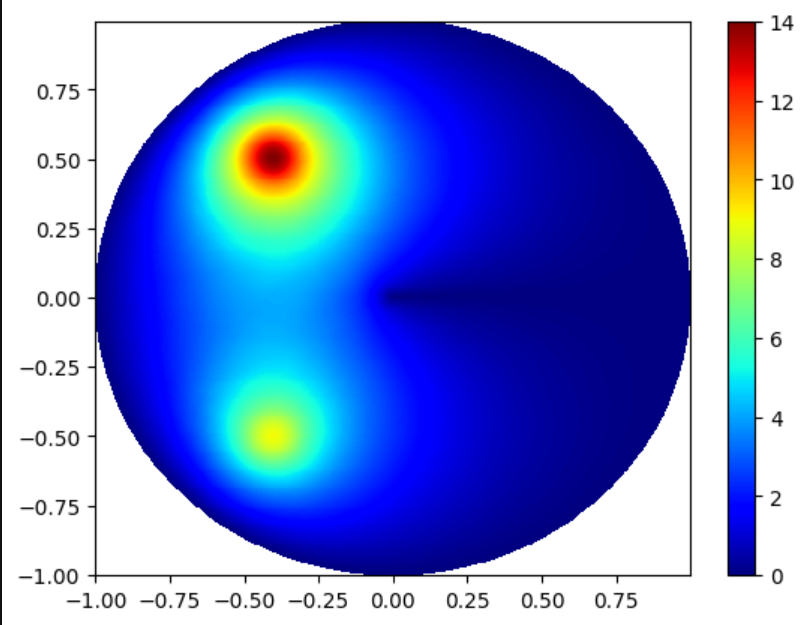}}
		{$u_{\vTheta}(x)$.}
		&
		\subf{\includegraphics[width=50mm]{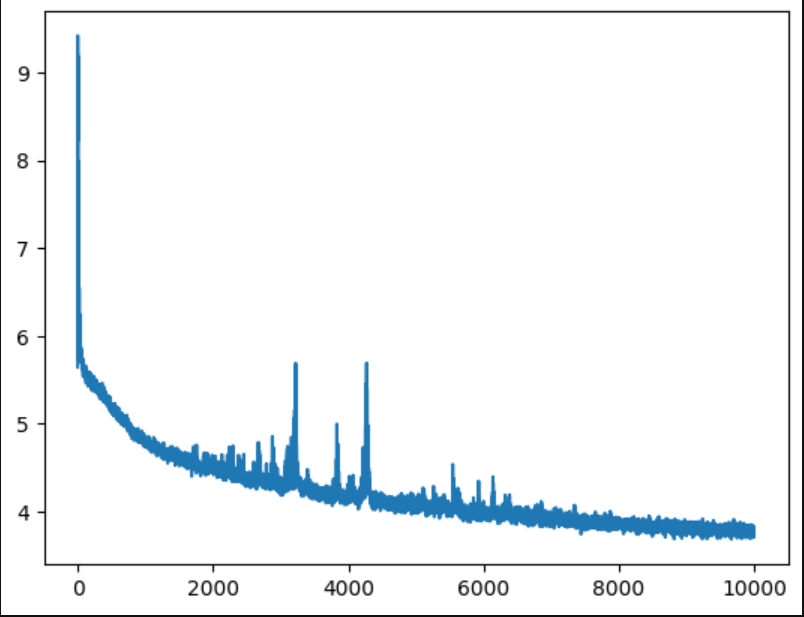}}
        {$\log(L_{N}(\vTheta))$ vs. iterations.}
		\\
		\hline
		\subf{\includegraphics[width=50mm]{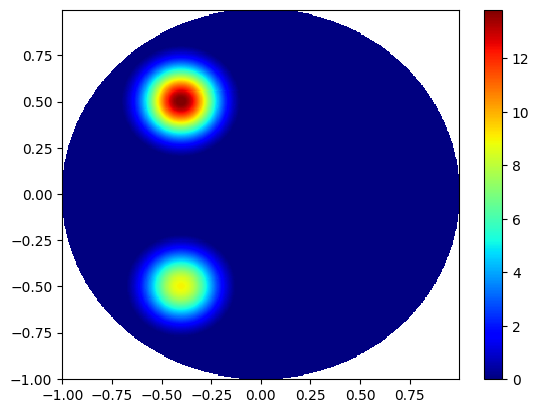}}
		{The obstacle $g(x)$.}
		&
		\subf{\includegraphics[width=50mm]{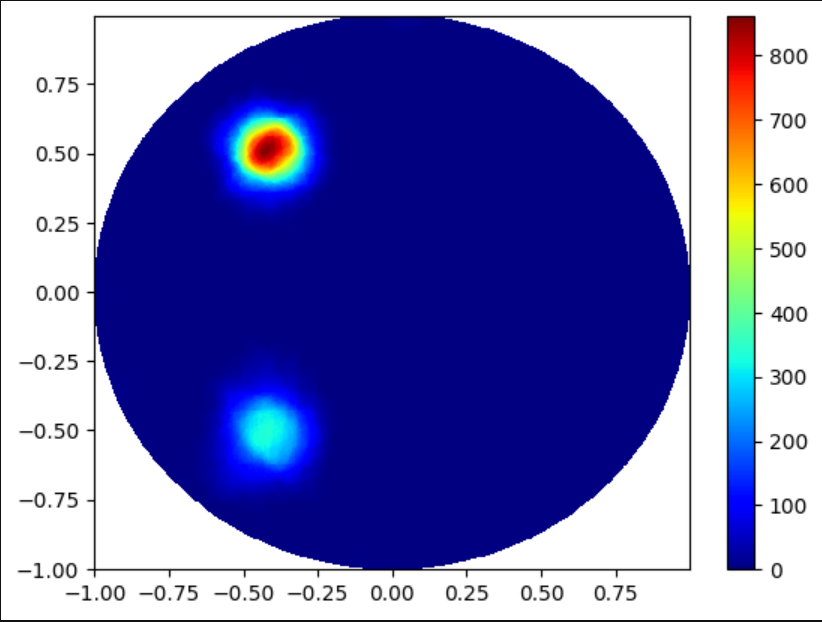}}
		{$\eta^2_{\vTheta}(x)$.}
		\\
		\hline
	\end{tabular}
	\caption{
	Computational solution $u_{\vTheta}$ (top left), evolution of the loss function (top right), the obstacle function $g$ (bottom left), and the computed contact set $\eta^2_{\vTheta}$ (bottom right). We used a learning rate $\ell = 0.001$, with 4000 collocation points, 10000 optimization steps, and 66654 degrees of freedom. In this case, the surrogate function $\dD(x)$ agrees with  $d(x,\partial \Omega)$.}
	\label{fig:sing_pert}
\end{figure}

\begin{figure}[h]
	\centering
        \includegraphics[width=80mm]{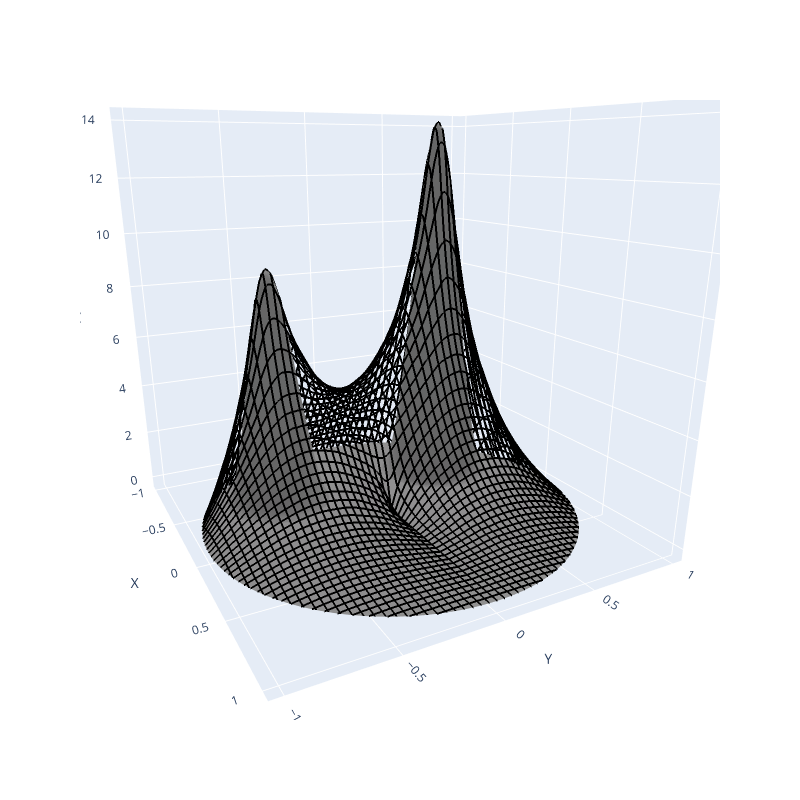}
	\caption{
	The computational solution $u_{\vTheta}$ (black wire-frame), and the obstacle function $g$ (gray), in Example \ref{ex:singularly-perturbed}.
	}
	\label{fig:sol_vs_obs}
\end{figure}

 \end{example}

%-----------------------------------------------------
%-----------------------------------------------------
\bibliographystyle{abbrv}
\bibliography{FOSLS.bib}
%-----------------------------------------------------
%-----------------------------------------------------
\end{document}